\documentclass[11pt, reqno]{amsart}
\usepackage{amssymb, amsmath, amsthm, mathrsfs, amsfonts, latexsym, tikz}
\usepackage[margin=1.25in]{geometry}

\newtheorem{theorem}{Theorem}[section]
\newtheorem{proposition}[theorem]{Proposition}
\newtheorem{lemma}[theorem]{Lemma}
\newtheorem{corollary}[theorem]{Corollary}

\newtheorem{remark}[theorem]{Remark}

\def\R{{\mathbb R}}
\def\Z{{\mathbb Z}}
\def\E{{\mathbb E}}
\def\P{{\mathbb P}}
\def\eps{\varepsilon}

\numberwithin{equation}{section}

\begin{document}

\title[Local times of anisotropic Gaussian random fields]
{Local times of anisotropic Gaussian random fields and stochastic heat equation}
\author{Cheuk Yin Lee \and Yimin Xiao}



\begin{abstract}
We study the local times of a large class of Gaussian random fields satisfying strong local nondeterminism with respect to an anisotropic metric.
We establish moment estimates and H\"{o}lder conditions for the local times of the Gaussian random fields. Our key estimates rely on geometric 
properties of Voronoi partitions with respect to an anisotropic metric and the use of Besicovitch's covering theorem. As a consequence, we 
deduce sample path properties of the Gaussian random fields that are related to Chung's law of the iterated logarithm and modulus of non-differentiability.
Moreover, we apply our results to systems of stochastic heat equations with additive Gaussian noise and determine the exact Hausdorff 
measure function with respect to the parabolic metric for the level sets of the solutions.
\end{abstract}

\maketitle

\section{Introduction}

The objective of this paper is to study the regularities of local times and related sample path properties of a large class of
Gaussian random fields that satisfy a property of strong local nondeterminism with respect to an anisotropic metric. 
The solution of systems 
of stochastic heat equations with additive Gaussian noise is an important example of these Gaussian random fields.

We consider a mean zero, continuous Gaussian random field $X = \{X(t), t \in \R^N\}$ with values in $\R^d$ defined by 
\begin{equation}\label{def:X}
X(t) = (X_1(t), \dots, X_d(t)),
\end{equation}
where $X_1, \dots, X_d$ are i.i.d.~copies of a real-valued Gaussian field $Y = \{ Y(t), t \in \R^N \}$.

Throughout this paper, let $T$ denote a closed (or open) bounded interval (rectangle) in $\R^N$.
We say that $X$ satisfies condition (A) on $T$ if there exists a constant vector $H = (H_1, \dots, H_N) \in (0, 1)^N$ 
such that the following conditions (A1) and (A2) hold.
\begin{itemize}
\item[(A1)]
There exist constants $0 < C_1 \le C_2 < \infty$ such that
\begin{equation*}\label{Eq:metric_rho}
C_1 \rho^2(t, s) \le \E[(Y(t) - Y(s))^2] \le C_2 \rho^2(t, s)
\end{equation*}
for all $t, s \in T$, where $\rho$ is the anisotropic metric defined by
\begin{align}\label{D:rho}
\rho(t, s) = \sum_{j=1}^N |t_j - s_j|^{H_j}.
\end{align}
\item[(A2)]
There exists a constant $C_3 > 0$ such that 
for all integers $n \ge 1$ and all $t, t^1, \dots, t^n \in T$,
\begin{equation*}\label{Eq:SLND_rho}
\mathrm{Var}(Y(t)|Y(t^1), \dots, Y(t^n)) \ge C_3 \min_{0 \le k \le n} \rho^2(t, t^k),
\end{equation*}
where $t^0 = 0$.
\end{itemize}

Condition (A2) is called the strong local nondeterminism in metric $\rho$.
The concept of local nondeterminism (LND) was first introduced by Berman \cite{B73} to study the existence 
and joint continuity of the local times of Gaussian processes. Berman's LND property was extended by Pitt \cite{P78} 
to  Gaussian random fields and  by Cuzick \cite{Cuzick78} to local
$\phi$-nondeterminism for an arbitrary positive function $\phi$. Moreover, strong LND and other forms of LND have 
been developed and they are useful for studying various sample path properties of Gaussian random fields (see, e.g., 
\cite{X06,X08} and references therein for more information).

It was proved by Pitt \cite{P78} that the multiparameter fractional Brownian motion with Hurst index $H \in (0, 1)$ 
satisfies strong LND with $\rho(t, s) = |t - s|^H$. The H\"{o}lder conditions for the local times and the Hausdorff 
measure of the level sets of Gaussian fields with stationary increments and strong LND were studied by Xiao \cite{X97}.
Similar results were obtained by Ayache et al.~\cite{AWX08} for fractional Brownian sheets, and by Wu and Xiao 
\cite{WX11} and Lee \cite{Lee22a, Lee22b} for other anisotropic Gaussian fields satisfying the property of
sectorial local nondeterminism. This latter property is weaker than the above property (A2) and was first
proved by Khoshnevisan and Xiao \cite{KX07} for the Brownian sheet, by Wu and Xiao \cite{WX07}
for fractional Brownian sheets, and Lee and Xiao \cite{LeeX19} for the solution of stochastic wave equations with 
Gaussian noise which is white in time and colored in space.

Conditions (A1) and (A2) above were considered in Xiao \cite{X09} to study various properties of sample paths 
and local times. In \cite{X09}, moment estimates for the local times on Euclidean balls were proved, but those 
estimates were not sharp enough to obtain exact regularity results and sharp sample path properties. 
In the current paper, we consider the local times of $X$ on anisotropic balls and strengthen the results by proving 
sharp moment estimates as well as  local and global H\"older conditions for the local times. In the case of 
anisotropic balls, the powers appearing in our bounds are different from those in \cite{X09, WX11,Lee22a, Lee22b}. 

In \cite{X09, LX12}, sufficient conditions for Gaussian random fields with stationary increments to satisfy condition (A) 
were obtained in terms of their spectral densities or spectral measures. One of our motivations is to provide a more
general framework for studying Gaussian random fields without stationary increments. We remark that the solution
$u= \{u(t, x), t\ge 0, x \in \R^N\}$ of a system of stochastic heat equations with additive Gaussian noise does not 
have stationary increments in the time-space variable $(t, x)$. For another example of Gaussian random fields 
without stationary increments that satisfies (A), see Lee and Xiao \cite{LeeX21}. In Section 5, we show that the solution
of a system of stochastic heat equations with additive Gaussian noise satisfies condition (A) as well as the 
conditions in Dalang et al. \cite{DMX17}.
Hence our results are  applicable to $u= \{u(t, x), t\ge 0, x \in \R^N\}$. Furthermore, we obtain a precise result 
regrading the Hausdorff measure function (in the metric $\rho$) for the level sets of  $u= \{u(t, x), t\ge 0, x \in \R^N\}$. 

The novelty of the current paper is that we employ two geometric tools to obtain our key estimates.
First, for any given points $t^1, \dots, t^n \in S \subset T$,  in order to estimate integrals of the form
\begin{align}\label{E:Intro_int}
\int_S \left[ \min_{1\le l \le n} \rho(t, t^l) \right]^{-\beta} dt,
\end{align}
 we consider the Voronoi partition
\begin{align*}
\Gamma_l = \left\{ t \in S : \tilde\rho(t, t^l) \le \tilde\rho(t, t^j) \text{ for all } j \in \{1, \dots, n\} \right\}, \quad l = 1, \dots, n,
\end{align*}
of $S$ generated by these fixed points, where $\tilde\rho(t, s) = \max_{1 \le j \le N}|t_j-s_j|^{H_j}$ is an anisotropic 
metric which is equivalent to $\rho$. We study the geometric properties of the Voronoi partition and prove that each 
$\Gamma_l$ (which contains $t^l$) satisfies an anisotropic star shape property at the point $t^l$ (see Lemma \ref{lem:Gamma}).
Our integral estimate for \eqref{E:Intro_int} hinges on this crucial geometric property of the Voronoi partition.
We refer to, e.g., \cite{K89, AKL13} for general theory, properties, and applications of Voronoi partitions and diagrams.

Secondly, as in \cite{X97}, finding moment estimates for the increments of the local times relies on estimates for the integrals
\begin{align*}
\int_{S^n} \prod_{j=1}^n \left[ \min_{0 \le l \le j-1} \rho(t^j, t^l)\right]^{-d} 
\prod_{j=1}^n \left[ \min_{0 \le l \le n, l \ne j} \rho(t^j, t^l)\right]^{-\gamma}  dt^1 \cdots dt^n.
\end{align*}
We point out that there was a gap in the proof of Lemma 2.5 in \cite{X97}, namely, the second inequality in (2.20) on page 140 
in \cite{X97} may not hold in general if $N \ge 2$. The missing part that is needed in order to complete the proof is the following: 
There exists a constant $K = K_N$ depending only on $N$ such that for arbitrary $n$ given points $ t^1, \dots, t^n$ and 
$0 \le i \le n$, we have
\begin{align*}
\# \big(\left\{ j \in \{1, \dots, n\} : \tilde\rho(t^j, t^i) \le \tilde\rho(t^j, t^l) \text{ for all } l \in \{0, 1, \dots, n\} \setminus \{j\} \right\}
\big) \le K_N,
\end{align*}
where $\# (J)$ denotes the cardinality of the set $J$ and, again, $t^0 = 0$. This universal bound is established in Lemma \ref{Lem:K}. 
Our idea is to relate this problem to covering points with 
balls in metric $\tilde\rho$ and then apply a version of  Besicovitch's covering theorem adapted to comparable intervals 
(see Lemma \ref{Lem:cov_thm}). This allows us to fill the gap in the proof of \cite{X97} and give a complete proof for 
the moment estimates for the increments of local times.

The rest of the paper is organized as follows. In Section \ref{sect:cont},  first we establish some auxiliary lemmas including 
Lemmas 2.2 and 2.6 about the anisotropic star shape property for the Voronoi partitions with respect to an anisotropic metric 
and a universal bound involving Besicovitch's covering theorem, respectively.
Then, we use these tools to derive sharp moment estimates for the local times 
in Proposition 2.4 and moment estimates for the increments of the local times in Proposition 2.8.
In Section \ref{sect:Holder_cond}, we use the moment estimates in Section \ref{sect:cont} to prove Theorem \ref{Th:Holder} which concerns local and global H\"{o}lder conditions for the local times, and deduce Theorem \ref{T:sample} which provides lower bounds for Chung's law of the iterated logarithm and the modulus of non-differentiability for the Gaussian random field.
Then in Section \ref{sect:level_set}, we discuss the Hausdorff dimension and Hausdorff measure of the level sets.
Finally in Section \ref{sect:SHE}, we apply the results to the solution of a system of stochastic heat equations and 
determine the exact Hausdorff measure function of the level sets with respect to the parabolic metric.

Throughout this paper, we let $C$ denote a constant whose value may be different in each appearance, and 
$C_1, C_2, K_1, K_2, \dots$ denote specific constants. We let $\lambda_N$ denote the Lebesgue measure on 
$\R^N$, and for any Borel set $S \subset \R^N$, let $\mathscr{B}(S)$ denote the $\sigma$-algebra of Borel subsets
of $S$. For $a \in \R^N$ and $r > 0$, let $B_\rho(a, r) = \{ t \in \R^N : \rho(t, a) \le r \}$ be the closed ball centered 
at $a$ with radius $r$ in the metric $\rho$. We denote a finite sequence of points in $\R^N$ by $t^1, t^2, \dots, t^n$, 
and for a given $k \in \{1, \dots, n\}$, the coordinates of the point $t^k$ are written as $t^k = (t^k_1, t^k_2, \dots, t^k_N)$.

\smallskip

\section{Moment estimates for the local times}
\label{sect:cont}

In this section, we study the joint continuity of local times for Gaussian random fields satisfying condition (A).
Let us recall the definition and properties of local times.
Let $X(t)$ be a vector field on $\R^N$ with values in $\R^d$ and $S \in \mathscr{B}(\R^N)$.
The occupation measure of $X$ on $S$ is the Borel measure on $\R^d$ defined by
\[\mu_S(B) = \lambda_N\{ t \in S : X(t) \in B \}, \quad B \in \mathscr{B}(\R^d).\]
We say that $X$ has a \emph{local time} on $S$ if the occupation measure $\mu_S$ is absolutely continuous 
with respect to the Lebesgue measure $\lambda_d$. In this case, the local time of $X$ is defined as the 
Radon--Nikodym derivative:
\[ L(x, S) = \frac{d\mu_S}{d\lambda_d}(x). \]
Note that if $X$ has a local time on $S$, then it also has a local time on any Borel set $A \subset S$.

By Theorem 6.3 of Geman and Horowitz \cite{GH80}, when the local time exists on $S$, one can choose a 
version of the local time, still denoted by $L(x, S)$, which is a kernel in the sense that
\begin{enumerate}
\item[(i)] $L(\cdot, A)$ is $\mathscr{B}(\R^d)$-measurable for each fixed $A \in \mathscr{B}(S)$; and
\item[(ii)] $L(x, \cdot)$ is a Borel measure on $\mathscr{B}(S)$ for each fixed $x \in \R^d$.
\end{enumerate}
Moreover, by Theorem 6.4 of \cite{GH80}, $L(x, S)$ satisfies the following \emph{occupation density formula}: 
for any nonnegative Borel function $f$ on $\R^d$,
\begin{equation}\label{Eq:ODF}
\int_S f(X(t)) \,dt = \int_{\R^d} f(x) L(x, S) \, dx.
\end{equation}

Let $T = \prod_{j=1}^N[\tau_j, \tau_j+h_j]$ be a compact interval, where $h_j > 0$ for $j = 1, \dots, N$.
We say that the local time is \emph{jointly continuous} on $T$ if we can find a version of the local time such 
that a.s.~$L\big(x, \prod_{j=1}^N [\tau_j, \tau_j + s_j] \big)$ is jointly continuous in all variables $(x, s)$ in $\R^d 
\times \prod_{j=1}^N[0, h_j]$.
Throughout this paper, we will always use a jointly continuous version of the local time whenever it exists.
When a local time is jointly continuous, it can be uniquely extended to a kernel and if, in addition, $X$ is 
continuous, $L(x, \cdot)$ defines a Borel measure supported on the level set $X^{-1}(x) \cap T = \{t \in T : 
X(t) = x \}$; see \cite[p.12, Remark (c)]{GH80} or \cite[Theorem 8.6.1]{Adler}.

Let $X$ be a Gaussian random field defined by \eqref{def:X}.
If condition (A1) is satisfied on $T$, then by Theorem 8.1 of Xiao \cite{X09}, $X$ has a local time on $T$ with 
$L(\cdot, T) \in L^2(\lambda_d \times \P)$ if any only if $d < \sum_{j=1}^N (1/H_j)$. Moreover, when the 
latter condition holds, the local time has the following representation in $L^2(\lambda_d \times \P)$: for 
any $S \in \mathscr{B}(T)$,
\begin{equation}\label{Eq:local_time_rep}
L(x, S) = (2\pi)^{-d} \int_{\R^d} du\, e^{-i\langle u, x \rangle} \int_S dt\, e^{i\langle u, X(t) \rangle}.
\end{equation}
It follows that  for any integer $n \ge 1$ and any $x \in \R^d$,
\begin{equation}\label{Eq:local_time_mom}
\E[L(x, S)^n] = (2\pi)^{-nd} \int_{\R^{nd}} d\bar{u} \int_{S^n} d\bar{t} \, e^{-i\sum_{j=1}^n \langle u^j, x\rangle}\, 
\E\left[e^{i\sum_{j=1}^n\langle u^j, X(t^j) \rangle}\right]
\end{equation}
and for any even integer $n \ge 2$ and $x, y \in \R^d$,
\begin{align}
\begin{aligned}\label{Eq:local_time_incre_mom}
\E[(L(x, S) - L(y, S))^n]
= (2\pi)^{-nd}  \int_{\R^{nd}} d\bar{u} \int_{S^n} d\bar{t}
\prod_{j=1}^n \left[e^{-i\langle u^j, x\rangle} - e^{-i\langle u^j, y\rangle}\right] \E\left[e^{i\sum_{l=1}^n\langle u^l, X(t^l)\rangle}\right],
\end{aligned}
\end{align}
where $\bar{u} = (u^1, \dots, u^n)$ and $\bar{t} = (t^1, \dots, t^n)$. See, e.g., \cite[\S 25]{GH80}.

In fact, under condition (A), the condition $d < \sum_{j=1}^N (1/H_j)$  implies not only the existence of local times but also  
their joint continuity. This has been proved in \cite[Theorem 8.2]{X09}:

\begin{theorem}\label{Th:joint_cont_LT}
Suppose $X$ satisfies condition \textup{(A)} on $T$ and $d < \sum_{j=1}^N (1/H_j)$. Then $X$ has a jointly 
continuous local time on $T$.
\end{theorem}

The proof of Theorem \ref{Th:joint_cont_LT} in \cite{X09} is based on moment estimates of the local times on the 
Euclidean balls and a multiparameter version of Kolmogorov's continuity theorem. However, those moment estimates 
are not sharp enough to deduce optimal H\"older conditions for the local times and the exact Hausdorff measure of the 
level sets.

The goal of this section is to prove sharp moment estimates for the local times on anisotropic balls. 
These estimates are provided in Propositions \ref{Prop:mom_LT} and \ref{Prop:mom_LT_incre} below.
We extend the method of Xiao \cite[Lemma 2.5]{X97}. 

From now on, we denote $Q :=  \sum_{j=1}^N (1/H_j)$. This quantity arises naturally in characterizing some fractal 
properties of the Gaussian random field $X$ that satisfies condition \textup{(A)} (see, e.g., \cite{X09}). It will appear in 
most of the theorems and lemmas in the present paper. 

To facilitate some of our arguments, let us define the metric
\begin{align}\label{D:rho_tilde}
\tilde{\rho}(t, s) = \max_{1 \le j \le N} |t_j - s_j|^{H_j}, \quad t, s \in \R^N,
\end{align}
which is equivalent to the metric $\rho$ defined in \eqref{D:rho}.
Indeed, we have
\begin{equation}\label{Eq:equiv_metric}
\tilde{\rho}(t,s) \le \rho(t, s) \le N \tilde{\rho}(t, s) \quad \text{for all } t, s \in \R^N.
\end{equation}
Let us begin with the following lemma, which is concerned with the geometric properties of the Voronoi partition 
generated by $m$ given points with respect to the anisotropic metric $\tilde\rho$.

\begin{lemma}\label{lem:Gamma}
Fix $m$ distinct points $t^1, \dots, t^m \in \R^N$.
For $l = 1, \dots, m$, define
\begin{align}\label{D:Gamma}
\Gamma_l = \left\{ t \in \R^N : \tilde\rho(t, t^l) = \min_{1 \le k \le m} \tilde\rho(t, t^k) \right\}.
\end{align}
Then the following properties hold:
\begin{enumerate}
\item[(i)] $\R^N = \bigcup_{l=1}^m \Gamma_l$ and $\lambda_N(\Gamma_l \cap \Gamma_{l'}) = 0$ whenever $l \ne l'$.
\item[(ii)] $\Gamma_l$ satisfies the following anisotropic star shape property at the point $t^l$:
\begin{align}\label{E:star}
t \in \Gamma_l \quad \text{implies} \quad t^l + \varepsilon^E (t-t^l) \in \Gamma_l\, \text{ for all }\varepsilon \in (0, 1),
\end{align}
where $E$ is the diagonal matrix $\operatorname{diag}(1/H_1, \dots, 1/H_N)$.
\end{enumerate}
\end{lemma}

\begin{proof}
(i). It is clear that the union of all $\Gamma_l$'s is $\R^N$. For $l \ne l'$, we have
\begin{align*}
\lambda_N(\Gamma_l \cap \Gamma_{l'}) &\le \lambda_N\left\{ t \in \R^N : \tilde\rho(t, t^l) = \tilde\rho(t, t^{l'})\right\}\\
& \le \sum_{i=1}^N \sum_{j=1}^N \lambda_N\left\{ t \in \R^N : |t_i-t^l_i|^{H_i} = |t_j - t^{l'}_j|^{H_j} \right\} = 0.
\end{align*}
(ii). It suffices to prove the property \eqref{E:star} for $l = 1$. Also, since
\begin{align*}
\Gamma_1 = \bigcap_{k=2}^m \left\{ t \in \R^N : \tilde\rho(t, t^1) \le \tilde\rho(t, t^k) \right\},
\end{align*}
we can further reduce the proof to the case $l = 1$ and $m = 2$.
With this reduction in mind, we assume $t \in \Gamma_1$, i.e.,
\begin{align}\label{E:t_in_Gamma}
\tilde\rho(t, t^1) \le \tilde\rho(t, t^2),
\end{align}
and aim to show that for all $\varepsilon \in (0, 1)$, the point $s = s(\varepsilon) := t^1 + \varepsilon^E (t-t^1)$ is in $\Gamma_1$, 
i.e.,
\begin{align}\label{E:s_in_Gamma}
\tilde\rho(s, t^1) \le \tilde\rho(s, t^2).
\end{align}
The points $t$, $t^1$, and $t^2$ may have different configurations.
Since only their relative positions are relevant to us, we may assume without loss of generality that $t^2$ is in the 
positive orthant relative to $t$, namely, $t^2_j \ge t_j$ for all $j \in \{ 1, \dots, N\}$; see Figure \ref{fig1}.
By the definition of the metric $\tilde\rho$ in \eqref{D:rho_tilde}, we have $\tilde\rho(t, t^2) = |t_j-t^2_j|^{H_j}$ for some 
$j \in \{1, \dots, N\}$. For simplicity, we assume that $j = 1$, i.e.,
\begin{align}\label{E:rho_t_t2}
\tilde\rho(t, t^2) = |t_1 - t^2_1|^{H_1} = (t^2_1 - t_1)^{H_1}
\end{align}
since the proof below also works for the other cases $j \ne 1$ in exactly the same way.
Note that
\begin{align}\label{E:rho_s_t1}
\tilde\rho(s, t^1) = \max_{1 \le j \le N} |t^1_j + \varepsilon^{1/H_j}(t_j-t^1_j) - t^1_j|^{H_j} = \varepsilon \tilde\rho(t, t^1)
\end{align}
and
\begin{align}\label{E:rho_s_t2}\begin{split}
\tilde\rho(s, t^2) \ge |s_1 + t^2_1|^{H_1}
& = |t^1_1+\varepsilon^{1/H_1}(t_1-t^1_1) - t^2_1|^{H_1}\\
& = |t_1-t^2_1 + (1-\varepsilon^{1/H_1}) (t^1_1 - t_1)|^{H_1}\\
& = |t^2_1 - t_1|^{H_1} \left| 1 - (1-\varepsilon^{1/H_1}) \frac{t^1_1-t_1}{t^2_1-t_1} \right|^{H_1}.
\end{split}\end{align}
In order to show \eqref{E:s_in_Gamma}, we consider two cases: (1) $t_1^1 \ge t_1$, and (2) $t_1^1 < t_1$; see Figure \ref{fig1}.

\begin{figure}
\begin{tikzpicture}
\draw [dashed] (0,-1) -- (0,3);
\draw [dashed] (-1,0) -- (4,0);
\draw [fill] (0, 0) circle [radius=.05] node [below left]{$t$};
\draw [fill] (3.3, 1.4) circle [radius=.05] node [above right]{$t^2$};
\draw [dashed] (3.3,1.4) -- (3.3,0);
\draw [dashed] (3.3,1.4) -- (0,1.4);
\draw [fill] (1.5, 2) circle [radius=.05] node [above]{$t^1$};
\draw [dashed] (1.5,2) -- (1.5,0);
\draw [dashed] (1.5,2) -- (0,2);
\end{tikzpicture}\qquad
\begin{tikzpicture}
\draw [dashed] (0,-1) -- (0,3);
\draw [dashed] (-2.5,0) -- (4,0);
\draw [fill] (0, 0) circle [radius=.05] node [below left]{$t$};
\draw [fill] (3.3, 1.4) circle [radius=.05] node [above right]{$t^2$};
\draw [dashed] (3.3,1.4) -- (3.3,0);
\draw [dashed] (3.3,1.4) -- (0,1.4);
\draw [fill] (-1.5, 2) circle [radius=.05] node [above]{$t^1$};
\draw [dashed] (-1.5,2) -- (-1.5,0);
\draw [dashed] (-1.5,2) -- (0,2);
\end{tikzpicture}
\caption{Two cases of configuration: (1) $t_1^1 \ge t_1$ (left), and (2) $t^1_1 < t_1$ (right)}\label{fig1}
\end{figure}
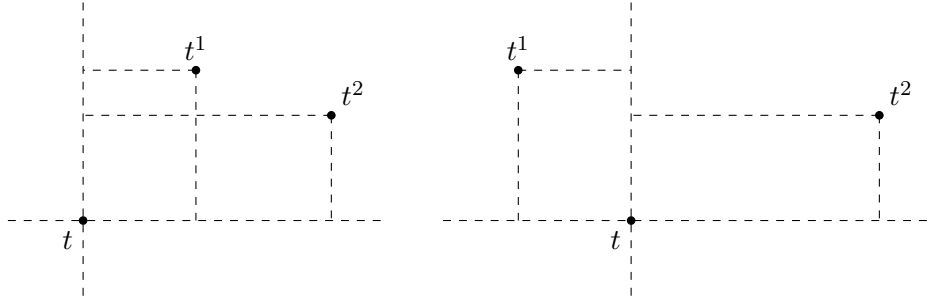

{\bf Case (1):} $t_1^1 \ge t_1$. 
By \eqref{E:t_in_Gamma} and \eqref{E:rho_t_t2}, $(t^1_1-t_1)^{H_1} \le \tilde\rho(t, t^1) \le \tilde\rho(t, t^2) = (t^2_1-t_1)^{H_1}$.
It follows that
\begin{align*}
0 \le \frac{t^1_1-t_1}{t^2_1-t_1} \le 1 \quad \text{and hence} \quad
\varepsilon^{1/H_1} \le 1 - (1-\varepsilon^{1/H_1}) \frac{t^1_1-t_1}{t^2_1-t_1} \le 1.
\end{align*}
This together with \eqref{E:rho_s_t2}, \eqref{E:rho_t_t2}, \eqref{E:t_in_Gamma} and \eqref{E:rho_s_t1} implies that
\begin{align*}
\tilde\rho(s, t^2) \ge \varepsilon|t^2_1-t_1|^{H_1} = \varepsilon\tilde\rho(t, t^2) 
\ge \varepsilon\tilde\rho(t, t^1) = \tilde\rho(s, t^1).
\end{align*}

{\bf Case (2):} $t^1_1 < t_1$. In this case, we have $t^2_1 - t_1 \ge 0$. It follows that
\begin{align*}
\frac{t^1_1-t_1}{t^2_1-t_1} < 0 \quad \text{and hence}  \quad
1 - (1-\varepsilon^{1/H_1}) \frac{t^1_1-t_1}{t^2_1-t_1} > 1 \ge \varepsilon^{1/H_1}.
\end{align*}
Again, this together with \eqref{E:rho_s_t2}, \eqref{E:rho_t_t2}, \eqref{E:t_in_Gamma} and \eqref{E:rho_s_t1} 
implies that $\tilde\rho(s, t^2) \ge \tilde\rho(s, t^1)$. This proves \eqref{E:s_in_Gamma} in both cases and finishes 
the proof of \eqref{E:star}.
\end{proof}

The next lemma is an integral estimate which will be useful later in the moment estimates of the local times.

\begin{lemma}\label{Lem:Bd_int}
Let $T \subset \R^N$ be any compact interval. Suppose that $0 < d \le \beta_0 < Q :=  \sum_{j=1}^N (1/H_j)$.
Then there is a finite constant $C = C(N, H, Q, \beta_0)$ such that for all intervals $S$ in $T$,  $\beta 
\in [d, \beta_0]$, integers $m \ge 2$, and distinct points $t^1, \dots, t^{m-1} \in S$, we have
\begin{equation}\label{Eq:Bd_int_Leb}
 \int_S \left[ \min_{0\le k \le m-1} \rho(t, t^k) \right]^{-\beta} dt \le C m^{\beta/Q} \lambda_N(S)^{1-\beta/Q}.
\end{equation}
In particular, for all $a \in \R^N$, $0 < r < 1$ and distinct $t^1, \dots, t^{m-1} \in B_\rho(a, r) \subset T$, we have
\begin{equation}\label{Eq:Bd_int}
 \int_{B_\rho(a, r)} \left[ \min_{0\le k \le m-1} \rho(t, t^k) \right]^{-\beta} dt \le C m^{\beta/Q} r^{Q-\beta}.
\end{equation}
\end{lemma}

\begin{proof}
Let $I$ denote the integral on the left-hand side of \eqref{Eq:Bd_int_Leb}.
Since $\lambda_N(S) = \lambda_N(\overline{S})$ for all intervals $S \subset T$, we may assume that $S$ is a closed interval.
We consider the Voronoi partition $\{\Gamma_l\}_{l=0}^{m-1}$ of $S$ generated by the points $t^0, t^1, \dots, t^{m-1}$ 
with respect to the metric $\tilde\rho$ defined in \eqref{D:rho_tilde}, i.e.,
\[\Gamma_l = \left\{ t \in S : \tilde\rho(t, t^l) = \min_{0 \le k \le m-1} \tilde\rho(t, t^k) \right\}, \quad l = 0, 1, \dots, m-1.\]
By \eqref{Eq:equiv_metric} and $S = \bigcup_{l=0}^{m-1} \Gamma_l$ (see Lemma \ref{lem:Gamma}), we have
\begin{align*}
I & \le  \int_S \left[ \min_{0\le k \le m-1} \tilde\rho(t, t^k) \right]^{-\beta} dt\\
&\le \sum_{l = 0}^{m-1} \int_{\Gamma_l} [\tilde\rho(t, t^l)]^{-\beta} dt
\le N^\beta \sum_{l = 0}^{m-1} \int_{\R^N} \Bigg(\sum_{j=1}^N|t_j - t^l_j|^{H_j}\Bigg)^{-\beta} {\bf 1}_{\Gamma_l}(t) \, dt.
\end{align*}
Fix $l \in \{0, 1, \dots, m-1\}$.
Let us consider the following anisotropic spherical coordinates $\Gamma_l$, namely, we let $t = t^l + h^E\Psi(\theta)$, 
where $E$ is the diagonal matrix $\operatorname{diag}(1/H_1, \dots, 1/H_N)$ and 
\begin{align*}
\Psi(\theta) = (\Psi_1(\theta), \dots, \Psi_N(\theta))^T: A \to \R^N, \quad \theta \in A := [0, 2\pi] \times [0, \pi]^{N-2},
\end{align*}
is defined by
\begin{align*}
\begin{cases}
\Psi_1(\theta) = [\cos(\theta_1)]^{2/H_1},\\
\Psi_2(\theta) = [\sin(\theta_1) \cos(\theta_2)]^{2/H_2},\\
\quad \vdots\\
\Psi_{N-1}(\theta) = [\sin(\theta_1) \dots \sin(\theta_{N-2}) \cos(\theta_{N-1})]^{2/H_{N-1}},\\
\Psi_N(\theta) = [\sin(\theta_1) \dots \sin(\theta_{N-2}) \sin(\theta_{N-1})]^{2/H_N}.
\end{cases}
\end{align*}
Here, $[x]^p := x |x|^{p-1}$ for any $x \in \R$.
The Jacobian $J_l$ for this transformation is well-defined and $|\det J_l| = h^{Q-1} \varphi(\theta)$, where $\varphi$ is a bounded 
nonnegative function. Under this change of coordinates, we have
\begin{align*}
\int_{\R^N} \Bigg(\sum_{j=1}^N|t_j - t^l_j|^{H_j}\Bigg)^{-\beta} {\bf 1}_{\Gamma_l}(t) \, dt
= \int_A d\theta \, \varphi(\theta) \int_0^\infty h^{Q-1-\beta} \, {\bf 1}_{\Gamma_l}(t^l+h^E\Psi(\theta))\, dh.
\end{align*}
The anisotropic star shape property of $\Gamma_l$ from Lemma \ref{lem:Gamma} shows that for any fixed $\theta \in A$, 
\begin{align*}
t^l + h^E \Psi(\theta) \in \Gamma_l \quad \text{implies} \quad t^l + (\varepsilon h)^E \Psi(\theta) \in \Gamma_l \text{ for all } 
\varepsilon \in (0, 1).
\end{align*}
Here, we have also used the fact that $S$ is a closed interval.
It follows that there exists $h_l(\theta) > 0$ such that ${\bf 1}_{\Gamma_l}(t^l+h^E\Psi(\theta)) = {\bf 1}_{(0, h_l(\theta)]}(h)$ and hence
\begin{align*}
I & \le N^\beta \sum_{l=0}^{m-1} \int_A d\theta \, \varphi(\theta) \int_0^{h_l(\theta)} h^{Q-1-\beta} dh = \frac{N^\beta}{Q-\beta} 
\sum_{l=0}^{m-1} \int_A [h_l(\theta)]^{Q-\beta} \varphi(\theta)\,d\theta.
\end{align*}
By the same variables change, the Lebesgue measure of $\Gamma_l$ can be computed as follows:
\begin{align}
\begin{aligned}\label{Eq:Leb_Gamma}
\lambda_N(\Gamma_l) 
= \int_A d\theta \, \varphi(\theta) \int_0^{h_l(\theta)} h^{Q-1} dh
= \frac{c_{N, H}}{Q} \int_A [h_l(\theta)]^Q \sigma(d\theta),
\end{aligned}
\end{align}
where the constant $c_{N, H} = \int_A \varphi(\theta) d\theta$ depends only on $N$ and $H$, and $\sigma(d\theta) = c_{N, H}^{-1}
 \varphi(\theta) d\theta$ is a probability measure on $A$.
Since $0 < \beta < Q$, the function $x \mapsto x^{1-\beta/Q}$ is concave on the interval $[0, \infty)$.
Then, we can apply Jensen's inequality for the probability measure $\sigma$ and the relation \eqref{Eq:Leb_Gamma} to obtain
\begin{align*}
I & \le \frac{c_{N, H} N^{\beta}}{Q-\beta} \sum_{l=0}^{m-1} \int_A \left\{[h_l(\theta)]^Q\right\}^{1-\beta/Q} \sigma(d\theta)\\
& \le \frac{c_{N, H} N^{\beta_0}}{Q-\beta_0}\sum_{l=0}^{m-1} \left\{\int_A [h_l(\theta)]^Q  \sigma(d\theta)\right\}^{1-\beta/Q}\\
& = \frac{c_{N, H} N^{\beta_0}}{Q-\beta_0}\sum_{l=0}^{m-1} \left\{ \frac{Q}{c_{N, H}} \lambda(\Gamma_l) \right\}^{1-\beta/Q}\\
& = C \sum_{l=0}^{m-1}  {\lambda(\Gamma_l)}^{1-\beta/Q},
\end{align*}
where $C$ is a constant that depends only on $N, H, Q$ and $\beta_0$. Then, by applying Jensen's inequality for the uniform 
probability measure on $\{0, 1, \dots, m-1\}$ and part (i) of Lemma \ref{lem:Gamma}, we get that
\begin{align*}
I & \le C m \left\{ \frac{1}{m} \sum_{l=0}^{m-1} \lambda(\Gamma_l) \right\}^{1-\beta/Q} = C m^{\beta/Q} \lambda(S)^{1-\beta/Q}.
\end{align*}
This proves \eqref{Eq:Bd_int_Leb}.
Finally, observe from \eqref{Eq:equiv_metric} that $B_\rho(a, r) \subset B_{\tilde\rho}(a, r)$ and the latter is a closed interval.
Hence, \eqref{Eq:Bd_int} follows from \eqref{Eq:Bd_int_Leb} by taking $S =  B_{\tilde\rho}(a, r)$ and using the fact that 
$\lambda_N(B_{\tilde\rho}(a, r)) = 2^N r^Q$.
\end{proof}

For any Gaussian vector $(Z_1, \dots, Z_n)$, let $\mathrm{Cov}(Z_1, \dots, Z_n)$ denote its covariance matrix. Recall that 
the determinant of this matrix can be evaluated by using the following formula (see e.g., \cite[Corollary A.2]{KM20}):
\begin{equation}\label{Eq:Cov_formula}
\det \mathrm{Cov}(Z_1, \dots, Z_n) = \mathrm{Var}(Z_1) \prod_{m=1}^n \mathrm{Var}(Z_m|Z_1, \dots, Z_{m-1}).
\end{equation}
The following proposition provides moment estimates for the local time.  

\begin{proposition}\label{Prop:mom_LT}
Suppose $X$ satisfies condition \textup{(A)} on $T$ and $d < Q$.  Then there exists a finite constant $C$ such that for all 
intervals $S$ in $T$, for all $x \in \R^d$ 
and all integers $n \ge 1$, we have
\[ \E[L(x, S)^n] \le C^n (n!)^{d/Q} \lambda_N(S)^{n(1-d/Q)}. \]
In particular, for all $a \in T$ and $r \in (0, 1)$ with $B_\rho(a, r) \subset T$, we have
\[ \E[L(x, B_\rho(a, r))^n] \le C^n (n!)^{d/Q} r^{n(Q-d)}. \]
\end{proposition}

\begin{proof}
By \eqref{Eq:local_time_mom}, we have
\begin{align*}
\E[L(x, S)^n] &= (2\pi)^{-nd} \int_{\R^{nd}} d\bar{u} \int_{S^n} d\bar{t} \, e^{-i\sum_{j=1}^n \langle u^j, x\rangle}\,
 \E\left[e^{i\sum_{j=1}^n\langle u^j, X(t^j) \rangle}\right],
\end{align*}
where $\bar{u} = (u^1, \dots, u^n)$ and $\bar{t} = (t^1, \dots, t^n)$.
Since $X_1, \dots, X_d$ are i.i.d.~copies of $Y$, we have
\begin{align*}
\E[L(x, S)^n] & \le (2\pi)^{-nd} \int_{S^n} d\bar{t} \prod_{k=1}^d \int_{\R^{n}} d\bar{u}_k  \, e^{-\frac{1}{2}
\mathrm{Var}(\sum_{j=1}^n u^j_k Y(t^j))}\\
& = (2\pi)^{-nd/2} \int_{S^n} \left[ \det \mathrm{Cov}(Y(t^1), \dots, Y(t^n)) \right]^{-d/2} d\bar{t},
\end{align*}
where $\bar{u}_k = (u^1_k, \dots, u^n_k)$ and the last equality follows from the connection between 
the characteristic function of a Gaussian random vector and the value of its density at the origin
(see Lemma \ref{Lem:CD82} below for a more general result). 
By \eqref{Eq:Cov_formula},
\[ \det\mathrm{Cov}(Y(t^1), \dots, Y(t^n)) = \mathrm{Var}(Y(t^1)) \prod_{m=2}^n \mathrm{Var}(Y(t^m)|Y(t^1), \dots, Y(t^{m-1})). 
\]
It follows from condition (A2) that
\begin{equation}\label{Eq:local_time_moment}
\E[L(x, S)^n]  \le C (2\pi)^{-nd/2} \int_{S^n} \prod_{m=1}^n \left[\min_{0 \le k \le m-1}\rho(t^m, t^k)\right]^{-d} d\bar{t}.
\end{equation}
If we integrate \eqref{Eq:local_time_moment} in the order of $dt^n, dt^{n-1}, \dots, dt^1$, and apply Lemma 
\ref{Lem:Bd_int} (with $\beta = d$) repeatedly, we deduce that
\begin{align*}
\E[L(x, S)^n] \le C^n (n!)^{d/Q} \lambda_N(S)^{n(1-d/Q)}.
\end{align*}
This completes the proof of the proposition.
\end{proof}

Next, we aim to prove moment estimates for the increments of the local time in both time and space variables. 
We are mostly following the approach of Lemma 2.5 in Xiao \cite{X97}, but let us point out that there is an 
error in the proof of (2.9) of that lemma: we observe that the second inequality in (2.20) of the proof may not 
be true if $N \ge 2$, because there may be multiple $j$'s such that $\min\{\xi(|t_{\pi(j)} - t_i|)^\gamma : 
i = 0 \textup{ or } i \ne \pi(j)\} 
= \xi(|t_{\pi(j)} - t_{\pi(1)}|)^\gamma$. [When $N=1$ and $\pi$ is the permutation such that $t_{\pi(1)}
 \le t_{\pi(2)}\le \ldots \le t_{\pi(n)}$, then (2.20) in \cite{X97} holds. When $N\ge 2$, the definition of the permutation $\pi$
 on page 139 is not sufficient for the second inequality in  (2.20) to hold.] Nevertheless, with the help of Lemma 
 \ref{Lem:K} provided below, we are able to give a complete proof for the moment estimates in Proposition 
 \ref{Prop:mom_LT_incre} below and therefore Lemma 2.5 in \cite{X97} can be corrected in a similar way.

The proof of Lemma \ref{Lem:K} below is based on Besicovitch's covering theorem.
The original theorem is stated for balls under the Euclidean metric (cf., e.g., \cite[p.30]{M}). 
In order to apply it in our anisotropic setting, we will use the following more general version of the covering 
theorem for comparable intervals in $\R^N$. 

\begin{lemma}\label{Lem:cov_thm}\cite[ Theorem 1.1 and Remark 5]{G75}
There exists a positive integer $M = M(N)$ depending only on $N$ with the following property.
For any bounded subset $A$ of $\R^N$ and any family $\mathscr{B} = \{Q(x) : x \in A \}$ of closed intervals 
such that $Q(x)$ is centered at $x$ for every $x \in A$ and, for every two points $x_1$ and $x_2$ of $A$, 
$Q(x_1)$ and $Q(x_2)$ can be translated to be concentric such  that one is contained in the other, 
there exists a sequence $\{Q_i\}$ in $\mathscr{B}$ such that:
\begin{enumerate}
\item[$(i)$] $A\subset \bigcup_i Q_i$;
\item[$(ii)$] the intervals of $\{Q_i\}$ can be divided into $M$ families of disjoint intervals.
\end{enumerate}
\end{lemma}

We use the Besicovitch covering theorem above to prove the next lemma.

\begin{lemma}\label{Lem:K}
There exists a positive integer $K = K(N)$ depending only on $N$ such that for any integer $n \ge 1$ and any 
distinct points $s^0, s^1, \dots, s^n \in \R^N$, the cardinality of the set of all $j \in \{1, \dots, n\}$ such that
\begin{equation}\label{Eq:min_rho_tilde}
\tilde{\rho}(s^j, s^0) = \min\{ \tilde{\rho}(s^j, s^i) : 0 \le i \le n, i \ne j \}
\end{equation}
is at most $K$, where $\tilde\rho$ is the equivalent metric defined in \eqref{D:rho_tilde}.
\end{lemma}

\begin{proof}
Without loss of generality, we may assume that \eqref{Eq:min_rho_tilde} is satisfied for $j = 1, \dots, k$.
Note that for $s \in \R^N$ and $r > 0$, the ball 
$B_{\tilde{\rho}}(s, r) := \{ t \in \R^N: \tilde{\rho}(t, s) \le r \}$ under the metric $\tilde{\rho}$ 
is the closed interval (hypercube) centered at $s$ with side lengths $2r^{1/H_1}, \dots, 2r^{1/H_N}$.
We will make use of Besicovitch's covering theorem (Lemma \ref{Lem:cov_thm}) to show that $k \le K$ for 
some positive integer $K = K(N)$ that depends only on the dimension $N$.

To this end, let 
\[ \delta_0 = \min\left\{\frac{\tilde{\rho}(s^i, s^0)}{\tilde{\rho}(s^j, s^0)} : i, j \in\{1, \dots, k\} \right\}. \]
Note that $0 < \delta_0 \le 1$.
Choose a small $0 < \varepsilon_0 < 1$ such that $(1- \varepsilon_0)^{1/H_p}(1+\delta_0^{1/H_p}) \ge 1$ for 
all $p \in \{1, \dots, N\}$.
Let $\varepsilon = \varepsilon_0 \min\{ \tilde{\rho}(s^i, s^0) : 1 \le i \le k \}$.
Let $A = \{s^1, \dots, s^k\}$ and
consider the family $\mathscr{B} = \{ B_{\tilde{\rho}}(s^1, r_1), \dots, B_{\tilde{\rho}}(s^k, r_k) \}$ of intervals, 
where $r_i = \tilde{\rho}(s^i, s^0)- \varepsilon$. 
By Besicovitch's covering theorem (Lemma \ref{Lem:cov_thm}), we can find sub-families $\mathscr{B}_1, 
\dots, \mathscr{B}_M \subset \mathscr{B}$, which we denote by $\mathscr{B}_i = \{ B_{\tilde{\rho}}(s^{i,1}, r_{i,1}), \dots, 
B_{\tilde{\rho}}(s^{i,J(i)}, r_{i,J(i)}) \}$, such that 
\[ A = \{ s^1, \dots, s^k \} \subset 
\bigcup_{i=1}^M \bigcup_{j=1}^{J(i)} B_{\tilde{\rho}}(s^{i,j}, r_{i,j}) \]
and for each sub-family $\mathscr{B}_i$, the intervals in $\mathscr{B}_i$ are pairwise disjoint, where $M = M(N)$ is 
a positive integer depending only on $N$. For each $1 \le j \le k$, by the assumption \eqref{Eq:min_rho_tilde}, 
if $\ell \ne j$, then $\tilde{\rho}(s^j, s^\ell) \ge \tilde{\rho}(s^j, s^0) > r_j$.
In other words, for each $1 \le j \le k$, the interval $B_{\tilde{\rho}}(s^j, r_j)$ does not contain any other $s^\ell$ with $\ell \ne j$. 
This means that at least $k$ intervals are needed to cover the set $A$, and hence $k \le J(1) + \dots + J(M)$.

Let us fix $i \in \{1, \ldots, M\}$ and estimate the cardinality $J(i)$ of the family $\mathscr{B}_i$. Consider the family $\mathscr{B}_i^* 
= \{ B_{\tilde{\rho}}(s^{i,1}, r^*_{i,1}), \dots, B_{\tilde{\rho}}(s^{i,J(i)}, r^*_{i, J(i)}) \}$, where $r^*_{i,j} = \tilde{\rho}(s^{i,j}, s^0)$.
Since the intervals in $\mathscr{B}_i$ are pairwise disjoint, this means that for any pair $s^{i,\ell} \ne s^{i,j}$ we 
can find some $p \in \{1, \dots, N\}$ such that 
\[ |s^{i,\ell}_p - s^{i,j}_p| > r_{i,\ell}^{1/H_p} + r_{i,j}^{1/H_p}. \]
Then, by the definitions of $r_i$, $\varepsilon$, and $\delta_0$, and by the choice of $\varepsilon_0$,
\begin{align*}
|s^{i,\ell}_p - s^{i,j}_p| &> (\tilde{\rho}(s^{i,\ell}, s^0)- \varepsilon)^{1/H_p} + 
(\tilde{\rho}(s^{i,j}, s^0)- \varepsilon)^{1/H_p}\\
& \ge (1-\varepsilon_0)^{1/H_p} \left(\tilde{\rho}(s^{i,\ell}, s^0)^{1/H_p} + 
\tilde{\rho}(s^{i,j}, s^0)^{1/H_p}\right)\\
& \ge (1-\varepsilon_0)^{1/H_p} (\delta_0^{1/H_p} + 1) \tilde{\rho}(s^{i,j}, s^0)^{1/H_p}\\
& \ge \tilde{\rho}(s^{i,j}, s^0)^{1/H_p}\\
& = (r_{i,j}^*)^{1/H_p}.
\end{align*}
It follows that $\tilde{\rho}(s^{i,\ell}, s^{i,j}) > r_{i,j}^*$, which means that the interval $B(s^{i,j}, r^*_{i,j})$ does not contain 
any other $s^{i,\ell}$ with $\ell \ne j$. On the other hand, every interval in $\mathscr{B}^*_i$ contains the point $s^0$, 
so these intervals are not pairwise disjoint. Then another application of Besicovitch's covering theorem (Lemma 
\ref{Lem:cov_thm}) to the set $\{ s^{i,1}, \dots, s^{i,J(i)} \}$ and the family $\mathscr{B}_i^*$ implies that $J(i) \le M$. 
Hence, we conclude that $k \le M^2$, and the proof is finished by taking $K = M^2$.
\end{proof}

Recall the following lemma in \cite[Lemma 2]{CD82}.

\begin{lemma}\label{Lem:CD82}
Let $Y_1, \dots, Y_n$ be mean zero Gaussian random variables that are linearly independent and assume that 
$\int_\R g(v) e^{-\varepsilon v^2} dv < \infty$ for all $\varepsilon > 0$. Then
\[ \int_{\R^n} g(v_1) \exp\bigg[ -\frac{1}{2} \mathrm{Var}\bigg( \sum_{l=1}^n v_l Y_l \bigg) \bigg] dv_1 \dots dv_n 
= \frac{(2\pi)^{(n-1)/2}}{\det\mathrm{Cov}(Y_1, \dots, Y_n)^{1/2}}
\int_\R g(v/\sigma_1) e^{-v^2/2} dv, \]
where $\sigma_1^2 = \mathrm{Var}(Y_1|Y_2, \dots, Y_n)$.
\end{lemma}

Recall also the equivalent metric $\tilde\rho$ defined in \eqref{D:rho_tilde}.
Note that condition (A2) implies
\begin{equation}\label{Eq:SLND_rho_tilde}
\mathrm{Var}(Y(t)| Y(t^1), \dots, Y(t^n)) \ge C_3 \min_{0 \le l \le n} \tilde{\rho}^2(t, t^l)
\end{equation}
for all $n \ge 1$ and all $t, t^1, \dots, t^n \in T$.

Now, we derive moment estimates for the increments of the local time.

\begin{proposition}\label{Prop:mom_LT_incre}
Suppose $X$ satisfies condition \textup{(A)} on $T$ and $d < Q$. 
Then there exist positive finite constants $C = C(T, N, d, H)$ and $K = K(N) > 1$ such that for all $\gamma \in (0, 1)$ 
small enough, for all intervals $S \subseteq T$,  $x, y \in \R^d$, and all even integers $n \ge 2$, we have
\[ \E[(L(x, S) - L(y, S))^n] \le C^n |x-y|^{n\gamma} (n!)^{d/Q + K\gamma/Q} \lambda_N(S)^{n(1-(d+\gamma)/Q)}. \]
In particular, for all $a \in T$, $0 < r < 1$ with $B_\rho(a, r) \subset T$, we have
\[ \E[(L(x, B_\rho(a, r)) - L(y, B_\rho(a, r)))^n] \le C^n |x-y|^{n\gamma} (n!)^{d/Q + K\gamma/Q} r^{n(Q-d-\gamma)}. \]
\end{proposition}

\begin{proof}
By \eqref{Eq:local_time_incre_mom}, for any even integer $n \ge 2$ and $x, y \in \R^d$,
\begin{align}\label{Eq:moment_LT_incre}
\begin{aligned}
\E[(L(x, S) - L(y, S))^n]
= (2\pi)^{-nd} \int_{S^n} d\bar{t} \int_{\R^{nd}} d\bar{u} 
\prod_{j=1}^n \left[e^{-i\langle u^j, x\rangle} - e^{-i\langle u^j, y\rangle}\right] \E\left[e^{i\sum_{l=1}^n\langle u^l, X(t^l)\rangle}\right],
\end{aligned}
\end{align}
where $\bar{u} = (u^1, \dots, u^n) \in \R^{nd}$ and $\bar{t} = (t^1, \dots, t^n) \in S^n$.
Note that
\begin{equation}\label{Eq:ch_f}
\E\left[e^{i\sum_{l=1}^n\langle u^l, X(t^l)\rangle}\right] = \exp\Bigg[{-\frac{1}{2}\mathrm{Var}\Bigg(\sum_{l=1}^n
\sum_{k=1}^d u_k^l X_k(t^l)\Bigg)}\Bigg]
\end{equation}
for all $u^1, \dots, u^n \in \R^d$ and $t^1, \dots, t^n \in S$.
For any $0 < \gamma < 1$, we have $|e^{iu}-1| \le 2^{1-\gamma} |u|^\gamma$
and $|u+v|^\gamma \le |u|^\gamma + |v|^\gamma$ for all $u, v \in \R$.
It follows that
\begin{equation}\label{Eq:prod_exp_diff}
\prod_{j=1}^n \left| e^{-i\langle u^j, x\rangle} - e^{-i\langle u^j, y\rangle} \right| 
\le 2^{(1-\gamma)n} |x-y|^{n\gamma} \sum_{\bar{k}} \prod_{j=1}^n {\big|u_{k_j}^j\big|}^\gamma
\end{equation}
for all $u^1, \dots, u^n, x, y \in \R^d$, where the summation runs over all 
$\bar{k} = (k_1, \dots, k_n) \in \{1, \dots, d\}^n$.
Then \eqref{Eq:moment_LT_incre}, \eqref{Eq:ch_f} and \eqref{Eq:prod_exp_diff} imply that
\begin{align}\label{Eq:Bd_LT_incre}
\E[(L(x, S) - L(y, S))^n] \le (2\pi)^{-nd}\, 2^n |x-y|^{n\gamma} \sum_{\bar{k}}\int_{S^n} J(\bar{t}, \bar{k})\,d\bar{t},
\end{align}
where 
\[ J(\bar{t}, \bar{k}) =  \int_{\R^{nd}} \Bigg( \prod_{j=1}^n {\big|u_{k_j}^j\big|}^\gamma \Bigg)
 \exp\Bigg[{-\frac{1}{2}\mathrm{Var}\Bigg(\sum_{l=1}^n\sum_{k=1}^d u_k^l X_k(t^l)\Bigg)}\Bigg] d\bar{u} \]
for $\bar{t} \in S^n$ and $\bar{k} \in \{1, \dots, d\}^n$. 
We may and will assume that $(t^1, \dots, t^n) \in S^n$ are distinct points because those which are not distinct form 
a set of Lebesgue measure 0. By the generalized H\"{o}lder inequality,
\begin{align*}
J(\bar{t}, \bar{k})
& \le \prod_{j=1}^n \Bigg\{ \int_{\R^{nd}} {\big|u_{k_j}^j\big|}^{n\gamma} 
\exp\Bigg[{-\frac{1}{2}\mathrm{Var}\Bigg(\sum_{l=1}^n\sum_{k=1}^d u_k^l X_k(t^l)\Bigg)}\Bigg] d\bar{u} \Bigg\}^{1/n}.
\end{align*}
Fix $\bar{t}$, $\bar{k}$, and $j$.
By condition (A2), the random variables $\{X_k(t^l) : 1 \le l \le n, 1 \le k \le d\}$ are linearly independent.
Then by Lemma \ref{Lem:CD82} and the fact that $X_1, \dots, X_d$ are i.i.d.~copies of $Y$,
\begin{align*}
\int_{\R^{nd}} {\big|u_{k_j}^j\big|}^{n\gamma} \exp\Bigg[{-\frac{1}{2}\mathrm{Var}\Bigg(\sum_{l=1}^n\sum_{k=1}^d u_k^l 
X_k(t^l)\Bigg)}\Bigg] d\bar{u}
= \frac{(2\pi)^{(nd-1)/2}}{\det \mathrm{Cov}(Y(t^1), \dots, Y(t^n))^{d/2}} 
\int_\R {\left|\frac{v}{\sigma_j}\right|}^{n\gamma} e^{-v^2/2} dv,
\end{align*}
where 
\begin{align*}
\sigma_j^2 = \mathrm{Var}\left(X_{k_j}(t^j) \,\big| \,X_k (t^l) : (k,l) \ne (k_j, j)\right) = \mathrm{Var}\left(Y(t^j)\, \big|\, Y(t^l) : l \ne j\right).
\end{align*}
By Jensen's inequality and Gaussian moment estimates,
\[ \int_\R |v|^{n\gamma} e^{-v^2/2} dv 
\le \sqrt{2\pi}\left( \int_\R  \frac{1}{\sqrt{2\pi}}|v|^n e^{-v^2/2} dv \right)^{\gamma} 
= \sqrt{2\pi} \left((n-1)!!\right)^\gamma \le \sqrt{2\pi} (n!)^\gamma. \]
It follows that
\begin{equation}\label{Eq:Bd_J}
 J(\bar{t}, \bar{k}) \le \frac{C^n (n!)^\gamma}{\det\mathrm{Cov}(Y(t^1), \dots, Y(t^n))^{d/2}} \prod_{j=1}^n \frac{1}{\sigma_j^\gamma}.
\end{equation}
By the covariance formula \eqref{Eq:Cov_formula} and condition (A2), or \eqref{Eq:SLND_rho_tilde},
\begin{align*}
\int_{S^n} J(\bar{t}, \bar{k}) \,d\bar{t} 
\le C^n (n!)^{\gamma} \int_{S^n} \prod_{j=1}^n \left[ \min_{0 \le l \le j-1} \tilde\rho(t^j, t^l)\right]^{-d} 
\prod_{j=1}^n \left[ \min_{0 \le l \le n, l \ne j} \tilde\rho(t^j, t^l)\right]^{-\gamma}  d\bar{t}.
\end{align*}
To estimate the integral over $S^n$, we consider, for any $j \in \{2, \dots, n\}$ and any given distinct points 
$t^0, t^1, \dots, t^{j-1}$, the Voronoi partition $\{ \Gamma_{j, i} \}_{i=0}^{j-1}$ of $S$ generated by the points 
$t^0, t^1, \dots, t^{j-1}$ with respect to the anisotropic metric $\tilde\rho$, namely,
\begin{align*}
\Gamma_{j, i} = \left\{ t^j \in S : \tilde\rho(t^j, t^i) = \min_{0 \le l \le j-1} \tilde\rho(t^j, t^l) \right\}, \quad 0 \le i \le j-1.
\end{align*}
Then
\begin{align*}
&\int_{S^n} \prod_{j=1}^n \left[ \min_{0 \le l \le j-1} \tilde\rho(t^j, t^l)\right]^{-d} \prod_{j=1}^n 
\left[ \min_{0 \le l \le n, l \ne j} \tilde\rho(t^j, t^l)\right]^{-\gamma}  d\bar{t}\\
& = \sum_{i_2=0}^1 \dots \sum_{i_n=0}^{n-1}\int_{S}  \int_{\Gamma_{2, i_2}} \dots  
\int_{\Gamma_{n, i_n}} d\bar{t} \, \prod_{j=1}^n \tilde\rho(t^j, t^{i_j})^{-d} \prod_{j=1}^{n}
 \left[ \min_{0 \le l \le n, l \ne j} \tilde\rho(t^j, t^l)\right]^{-\gamma},
\end{align*}
where $i_1 = 0$.

Let $(t^1, t^2, \dots, t^n) \in S \times \Gamma_{2, i_2} \times \cdots \times \Gamma_{n,i_n}$.
For each $1 \le j \le n$, let $\alpha_n(j)$ be the largest index in $\{0, 1, \dots, n\} \setminus \{j\}$ such that
\[ \tilde\rho(t^j, t^{\alpha_n(j)}) = \min_{0 \le l \le n, l \ne j} \tilde\rho(t^j, t^l). \]
Let $m_n$ be the number of $j$'s in $\{1, \dots, n\}$ such that $\alpha_n(j) = n$.
By Lemma \ref{Lem:K}, we have $0 \le m_n \le K$, where $K$ is a universal bound depending only on $N$.
Then
\begin{align*}
\prod_{j=1}^{n} \min_{0 \le l \le n, l \ne j} \tilde\rho(t^j, t^l) 
&= \prod_{\substack{1\le j < n\\ \alpha_n(j) = n}}\tilde\rho(t^j, t^n) \prod_{\substack{1\le j \le n\\ \alpha_n(j) \ne n}}
 \min_{0 \le l \le n, l \ne j}\tilde\rho(t^j, t^l)\\
& \ge \left[\tilde\rho(t^n, t^{i_n})\right]^{m_n} \prod_{\substack{1\le j \le n\\ \alpha_n(j) \ne n}}\min_{0 \le l \le n, l \ne j}\tilde\rho(t^j, t^l),
\end{align*}
where in the last inequality, we have used the fact that $t^n \in \Gamma_{n, i_n}$ implies $ \tilde\rho(t^j, t^n)\ge \tilde\rho(t^n, t^{i_n})$ 
for all $j$.
Next, for $1 \le j \le n$ with $\alpha_n(j) \ne n$, we consider the largest index $\alpha_{n-1}(j)$ in $\{0, 1, \dots, n\} \setminus\{j\}$ 
such that
\[ 
\tilde\rho(t^j, t^{\alpha_{n-1}(j)}) = \min_{0 \le l \le n, l \ne j} \tilde\rho(t^j, t^l) 
\]
and the number $m_{n-1}$ of $j$'s in $\{1, \dots, n\}$ such that $\alpha_n(j) \ne n$ and $\alpha_{n-1}(j) = n-1$.
Inductively, we can find integers $m_{n-1}, \dots, m_2, m_1$, with $0 \le m_j \le K$ for all $j$, such that
\[m_1 + \dots + m_n = n \quad \text{and} \quad
\prod_{j=1}^{n} \min_{0 \le l \le n, l \ne j} \tilde\rho(t^j, t^l) \ge \prod_{j=1}^n \left[\tilde\rho(t^j, t^{i_j})\right]^{m_j}. \]
It follows that
\begin{align*}
&\int_{S^n} \prod_{j=1}^n \left[ \min_{0 \le l \le j-1} \tilde\rho(t^j, t^l)\right]^{-d} \prod_{j=1}^n 
\left[ \min_{0 \le l \le n, l \ne j} \tilde\rho(t^j, t^l)\right]^{-\gamma}  d\bar{t}\\
& \le \sum_{(m_1, \dots, m_n) \in {\mathcal M}} \,
\sum_{i_2=0}^1 \dots \sum_{i_n=0}^{n-1}\int_{S}  \int_{\Gamma_{2, i_2}} \dots\int_{\Gamma_{n, i_n}} d\bar{t} \, 
\prod_{j=1}^{n} \left[\tilde\rho(t^j, t^{i_j})\right]^{-d-m_j \gamma},
\end{align*}
where 
\[
{\mathcal M} = \big\{ (m_1, \dots, m_n) \in \{0, 1, \dots, K\}^n : m_1 + \dots + m_n = n \big\}.
\]
The cardinality of ${\mathcal M}$ is bounded by $(K+1)^n$.
Fix any $\beta_0 \in (d, Q)$. Then for all $\gamma \in (0, 1)$ small enough, we have
\[ d < d + K\gamma < \beta_0 < Q. \]
If we fix $(t^1, t^2, \dots, t^{n-1}) \in S \times \Gamma_{2, i_2} \times \cdots \times \Gamma_{n-1, i_{n-1}}$, 
then by \eqref{Eq:equiv_metric} and Lemma \ref{Lem:Bd_int},
\begin{align*}
\sum_{i_n=0}^{n-1}\int_{\Gamma_{n, i_n}} \left[\rho(t^n, t^{i_n})\right]^{-d-m_n\gamma} \, dt^n
\le C n^{d/Q+m_n\gamma/Q} \lambda_N(S)^{1-d/Q-m_n\gamma/Q}.
\end{align*}
Inductively, we integrate in the order of $dt^{n-1}, \dots, dt^1$ to deduce that
\begin{align*}
&\int_{S^n} \prod_{j=1}^n \Big[ \min_{0 \le l \le j-1} \tilde\rho(t^j, t^l)\Big]^{-d} \prod_{j=1}^n 
\Big[ \min_{0 \le l \le n, l \ne j} \tilde\rho(t^j, t^l)\Big]^{-\gamma}  d\bar{t}\\
&\le C^n \sum_{(m_1, \dots, m_n) \in M}\, \prod_{j=1}^n \Big[ j\,^{d/Q + K\gamma/Q} \Big] 
\lambda_N(S)^{n(1-d/Q) - (m_1+\dots+m_n)\gamma/Q}\\
& \le C^n (K+1)^n (n!)^{d/Q+K\gamma/Q} \lambda_N(S)^{n(1-(d+\gamma)/Q)}.
\end{align*}
Therefore, 
\begin{equation}\label{Eq:int_D^n}
\int_{S^n} J(\bar{t}, \bar{k}) \, d\bar{t} \le C^n (n!)^{d/Q+K\gamma/Q} \lambda_N(S)^{n(1-(d+\gamma)/Q)}.
\end{equation}
Note that this bound does not depend on $\bar{k}$.
Combining \eqref{Eq:Bd_LT_incre} and \eqref{Eq:int_D^n}, we have
\[ \E[(L(x, S) - L(y, S))^n] \le C^n |x-y|^{n\gamma} (n!)^{d/Q + K\gamma/Q} \lambda_N(S)^{n(1-(d+\gamma)/Q)}. \]
This completes the proof of Proposition \ref{Prop:mom_LT_incre}.
\end{proof}

We end this section with some lemmas, which will be useful later.

\begin{lemma}\label{LT:large_dev}
Suppose $X$ satisfies condition \textup{(A)} on $T$ and $d < Q$. 
For any $b > 0$, there exists a finite constant $c$ such that the following hold.
\begin{enumerate}
\item[(i)] For all $a \in T$ and $0 < r < 1$ with $D:=B_\rho(a, r) \subset T$, for all $x \in \R^d$ and $u > 1$,
\begin{equation}\label{Lem:Eq1}
\P\left\{ L(x, D) \ge c\, r^{Q-d} u^{d/Q} \right\} \le \exp(-bu).
\end{equation}
\item[(ii)] For all $\gamma \in (0, 1)$ small enough, for all $a \in T$ and $0 < r < 1$ with 
$D:=B_\rho(a, r) \subset T$, for all $x, y \in \R^d$ with $|x-y|\le 1$, for all $u > 1$,
\begin{align}
\begin{aligned}\label{Lem:Eq2}
\P\left\{|L(x, D) - L(y, D)| \ge c\,|x-y|^{\gamma} r^{Q-d-\gamma}u^{d/Q + K\gamma/Q}\right\} \le \exp(-bu),
\end{aligned}
\end{align}
where $K = K(N) > 1$ is a finite constant.
\end{enumerate}
\end{lemma}

\begin{proof}
Let $A > 0$ be a constant. By Chebyshev's inequality and Proposition \ref{Prop:mom_LT},
\begin{align*}
\P\left\{ L(x, D) \ge A r^{Q-d} n^{d/Q} \right\} 
\le \frac{\E[L(x, D)^n]}{(A r^{Q-d} n^{d/Q})^n}
\le \left(\frac{C}{A}\right)^n \left( \frac{n!}{n^n}\right)^{d/Q}.
\end{align*}
By Stirling's formula, given $b > 0$, we can choose $A$ large enough such that for all $n \ge 1$,
\[ \P\left\{ L(x, D) \ge A r^{Q-d} n^{d/Q} \right\} \le \exp(-2b n). \]
This implies \eqref{Lem:Eq1}.
Similarly, \eqref{Lem:Eq2} can be proved by using Proposition \ref{Prop:mom_LT_incre}.
\end{proof}

\begin{lemma}\label{Lem1}
Suppose $X$ satisfies condition \textup{(A)} on $T$ and $d < Q$.
Then, there exists a positive finite constant $C$ such that the following statements hold.
\begin{enumerate}
\item[(i)] For all $a \in T$ and $0 < r < 1$ with $D:=B_\rho(a, r) \subset T$, for all $x \in \R^d$, for all integers $n \ge 1$, 
\begin{equation}\label{Lem1:Eq1}
\E[L(x + X(a), D)^n] \le C^n (n!)^{d/Q} r^{n(Q-d)}.
\end{equation} 
\item[(ii)] For all $\gamma \in (0, 1)$ small enough, for all $a \in T$ and $0 < r < 1$ with $D:=B_\rho(a, r) \subset T$, 
for all $x, y \in \R^d$ with $|x-y|\le 1$, for all even integers $n \ge 2$, 
\begin{align}
\begin{aligned}\label{Lem1:Eq2}
\E[(L(x + X(a), D) - L(y + X(a), D))^n]
\le C^n |x-y|^{n\gamma} (n!)^{d/Q + K\gamma/Q} r^{n(Q-d-\gamma)},
\end{aligned}
\end{align}
where $K = K(N) > 1$ is a finite constant.
\end{enumerate}
\end{lemma}

\begin{proof}
For any fixed $a \in T$, consider the Gaussian random field $\tilde{X}(t) = X(t) - X(a)$.
If $X$ has a local time $L(x, S)$ on $S$, then $\tilde{X}$ also has a local time $\tilde{L}(x, S)$ on $S$, and it follows 
from \eqref{Eq:ODF} that $\tilde{L}(x, S) = L(x + X(a), S)$. Moreover, $\tilde{X}$ satisfies 
Condition (A1) and a slightly different version of condition (A2) with the inequality
\[ \mathrm{Var}(\tilde{Y}(t)|\tilde{Y}(t^1), \dots, \tilde{Y}(t^n)) \ge C_3 \min\{ \rho^2(t, t^k): t \in \{ a, t^0, t^1, \dots, t^n \} \}. \]
With little modification for Lemma \ref{Lem:Bd_int}, the proofs of Proposition \ref{Prop:mom_LT} and \ref{Prop:mom_LT_incre} 
can be carried over to the Gaussian field $\tilde{X}$ to obtain \eqref{Lem1:Eq1} and \eqref{Lem1:Eq2}.
\end{proof}

\begin{lemma}\label{Lem2}
Suppose $X$ satisfies condition \textup{(A)} on $T$ and $d < Q$. 
For any $b > 0$, there exists a finite constant $c$ such that the following hold.
\begin{enumerate}
\item[(i)] For all $a \in T$ and $0 < r < 1$ with $D:=B_\rho(a, r) \subset T$, for all $x \in \R^d$ and $u > 1$,
\begin{equation}\label{Lem2:Eq1}
\P\left\{ L(x + X(a), D) \ge c\, r^{Q-d} u^{d/Q} \right\} \le \exp(-bu).
\end{equation}
\item[(ii)] For all $\gamma \in (0, 1)$ small enough, for all $a \in T$ and $0 < r < 1$ with $D:=B_\rho(a, r) \subset T$, 
for all $x, y \in \R^d$ with $|x-y|\le 1$, for all $u > 1$,
\begin{align}
\begin{aligned}\label{Lem2:Eq2}
\P\Big\{|L(x + X(a)&, D) - L(y + X(a), D)|
\ge c\, |x-y|^{\gamma} r^{Q-d-\gamma}u^{d/Q + K\gamma/Q}\Big\} \le \exp(-bu),
\end{aligned}
\end{align}
where $K = K(N) > 1$ is a finite constant.
\end{enumerate}
\end{lemma}

\begin{proof}
As in Lemma \ref{LT:large_dev}, the proof is based on Lemma \ref{Lem1} and Chebyshev's inequality.
\end{proof}

\section{H\"{o}lder conditions for the local times}
\label{sect:Holder_cond}

In this section, we study local and global H\"older conditions in the set variable for the local times, and discuss related 
sample path properties including Chung's law of the iterated logarithm and the modulus of non-differentiability.

From Lemma \ref{LT:large_dev} and the Borel--Cantelli lemma, one can easily derive the following law of the iterated 
logarithm for the local time $L(x, \cdot)$: there exists a finite constant $C$ such that for any $x \in \R^d$ and $t \in T$, 
\begin{equation}\label{LIL:LT}
\limsup_{r \to 0} \frac{L(x, B_\rho(t, r))}{\varphi(r)} \le C \quad \text{a.s.,} \quad \text{where} \quad \varphi(r) = 
r^{Q-d}(\log\log(1/r))^{d/Q}.
\end{equation}
Recall that $Q = \sum_{j=1}^N (1/H_j)$.

It follows from Fubini's theorem that with probability one, \eqref{LIL:LT} holds for $\lambda_N$-almost every $t \in T$.
Next, we are going to prove a stronger version of this result, which will be useful later for determining the exact 
Hausdorff measure of the level sets of $X$.

It follows from Theorem \ref{Th:joint_cont_LT} that if $X$ satisfies condition (A) on $T$ and $d < Q$, 
then $X$ has a jointly continuous local time on $T$. 
From now on, we always use the jointly continuous version for the local time whenever it exists, and still denote 
it by $L(x, \cdot)$.

\begin{theorem}\label{Thm:Holder_LT}
Suppose $X$ satisfies condition \textup{(A)} on $T$ and $d < Q$.
Then there exists a finite constant $C$ such that for any $x \in \R^d$,
with probability 1, 
\begin{equation}\label{Eq:Holder_LT}
\limsup_{r \to 0} \frac{L(x, B_\rho(t, r))}{\varphi(r)} \le C
\end{equation}
for $L(x, \cdot)$-almost every $t \in T$, where $\varphi(r) = r^{Q-d}(\log\log(1/r))^{d/Q}$.
\end{theorem}

\begin{proof}
The proof is similar to that of Proposition 4.1 in \cite{X97} and Theorem 8.10 in \cite{X09}.
For any $x \in \R^d$ and any integer $k \ge 1$, consider the random measure $L_k(x, \cdot)$ on Borel subsets 
$B$ of $T$ defined by
\begin{align}
\begin{aligned}\label{Def:L_k}
L_k(x, B) &= \int_B \left(\frac{k}{2\pi}\right)^{d/2} \exp\left({-\frac{k|X(t) - x|^2}{2}}\right) \, dt\\
& = \int_B \int_{\R^d} \frac{1}{(2\pi)^d} \exp\left( -\frac{|u|^2}{2k} + i\langle u, X(t) - x\rangle \right) du\, dt.
\end{aligned}
\end{align}
By equation \eqref{Eq:local_time_rep}, we see that a.s., for all $B$, $L_k(x, B) \to L(x, B)$ as $k \to \infty$.
It follows that a.s., $L_k(x, \cdot)$ converges weakly to $L(x, \cdot)$.

For each $m \ge 1$, define $f_m(t) = L(x, B_\rho(t, 2^{-m}))$.
By Propositions \ref{Prop:mom_LT} and \ref{Prop:mom_LT_incre}, and the multiparameter version of Kolmogorov's 
continuity theorem \cite{K}, $f_m(t)$ is a.s.~bounded and continuous on $T$.
Then by the a.s.~weak convergence of $L_k(x, \cdot)$, for all $m, n \ge 1$,
\[ \int_T\, [f_m(t)]^n L(x, dt) = \lim_{k \to \infty} \int_T \, [f_m(t)]^n L_{k}(x, dt) \quad \text{a.s.} \]
Hence, by the dominated convergence theorem, \eqref{Def:L_k} and \eqref{Eq:local_time_mom}, we have
\begin{align*}
&\quad \E \int_T [L(x, B_\rho(t, 2^{-m}))]^n L(x, dt)\\
& =\frac{1}{(2\pi)^d} \lim_{k \to \infty} \E \int_T dt \int_{\R^d} du \, \exp\left( -\frac{|u|^2}{2k} + i \langle u, 
X(t) - x\rangle \right) \left[L(x, B(t, 2^{-m})) \right]^n\\
& = \frac{1}{(2\pi)^d}\int_T dt \int_{\R^d} du \, \E \big[e^{i \langle u, X(t) - x\rangle} L(x, B(t, 2^{-m}))^n \big]\\
& = \frac{1}{(2\pi)^{(n+1)d}} \int_T \int_{{B_\rho(t, 2^{-m})}^n}d\bar{s} \int_{\R^{(n+1)d}} d\bar{u}\, e^{-i\sum_{\ell=1}^{n+1} 
\langle x, u^\ell \rangle} \E \left(e^{i\sum_{\ell=1}^{n+1} \langle u^\ell, X(s^\ell)\rangle}\right),
\end{align*}
where $\bar{u} = (u^1, \dots, u^{n+1}) \in \R^{(n+1)d}$, $\bar{s} = (t, s^2,\dots, s^{n+1})$ and $s^1 = t$.
Similar to the proof of Proposition \ref{Prop:mom_LT}, we can deduce that
\begin{align}\begin{aligned}\label{Eq:E_integral_f}
&\quad \E\int_T [L(x, B_\rho(t, 2^{-m}))]^n L(x, dt)\\
& \le C^n \int_{T \times {B_\rho(t, 2^{-m})}^n} \frac{d\bar{s}}{[\det \mathrm{Cov}(Y(t), Y(s^2), \dots, 
Y(s^{n+1}))]^{d/2}}\\
& \le C^n (n!)^{d/Q} 2^{-nm (Q-d)}.
\end{aligned}\end{align}

Let $A > 0$ be a constant whose value will be determined. Consider the random set
\[ B_m = \{ t \in T : L(x, B_\rho(t, 2^{-m})) > A \varphi(2^{-m}) \}. \]
Consider the random measure $\mu$ on $T$ defined by 
$\mu(B) = L(x, B)$ for any $B \in \mathscr{B}(T)$.
Take $n = \lfloor \log m \rfloor$, the integer part of $\log m$.
Then by \eqref{Eq:E_integral_f} and Stirling's formula,
\begin{align*}
\E\,\mu(B_m) &\le \frac{\E\int_T [L(x, B_\rho(t, 2^{-m}))]^n L(x, dt)}{[A\varphi(2^{-m})]^n}\\
& \le \frac{C^n(n!)^{d/Q}2^{-nm(Q-d)}}{A^n 2^{-nm(Q-d)}(\log m)^{nd/Q}} \le m^{-2}
\end{align*}
provided $A > 0$ is chosen large enough. This implies that
\[ \E \sum_{m=1}^\infty \mu(B_m) < \infty. \]
By the Borel--Cantelli lemma, with probability 1, for $\mu$-a.e.~$t \in T$, we have
\begin{equation}\label{Eq:Holder_LT_m}
\limsup_{m \to \infty} \frac{L(x, B_\rho(t, 2^{-m}))}{\varphi(2^{-m})} \le A.
\end{equation}
For any $r > 0$ small enough, there exists an integer $m$ such that $2^{-m} \le r < 2^{-m+1}$ and 
\eqref{Eq:Holder_LT_m}  can be applied. Since $\varphi(r)$ is increasing near $r = 0$, we can use 
a monotonicity argument to obtain \eqref{Eq:Holder_LT}.
\end{proof}

Next, we study the local and global H\"older conditions for $L^*$, the supremum of the local time 
defined by
\[ L^*(B) = \sup_{x \in \R^d} L(x, B), \quad B \in \mathscr{B}(\R^N). \]

\begin{theorem}\label{Th:Holder}
Suppose $X$ satisfies condition \textup{(A)} on $T$ and $d < Q$. 
Then there exist finite constants $C$ and $C'$ such that for any $t \in T$,
\begin{equation}\label{Eq:Holder1}
  \limsup_{r \to 0}  \frac{L^*(B_\rho(t, r))}{\varphi(r)} \le C \quad \text{a.s.,} \quad \text{where} \quad \varphi(r) = r^{Q-d}(\log\log(1/r))^{d/Q},
\end{equation}
and 
\begin{equation}\label{Eq:Holder2}
\limsup_{r \to 0} \sup_{t \in T} \frac{L^*(B_\rho(t, r))}{\Phi(r)} \le C' \quad \text{a.s.,} \quad \text{where} \quad \Phi(r) = r^{Q-d}(\log(1/r))^{d/Q}.
\end{equation}
\end{theorem}


\begin{proof}
As in \cite{E81, X97}, the proof of Theorem \ref{Th:Holder} is based on Lemma \ref{Lem2} and a 
chaining argument. We will give a sketch of the proof with necessary modifications.

In order to prove \eqref{Eq:Holder1}, it suffices to show that for any $a \in T$, 
\begin{equation}\label{Holder1}
\limsup_{n \to \infty} \frac{L^*(B_n)}{\varphi(2^{-n})} \le C \quad \textup{a.s.}
\end{equation}
where $B_n = B_\rho(a, 2^{-n})$.
We divide the proof of \eqref{Holder1} into four steps.

(1). By Lemma 2.1 of Talagrand \cite{T95}, there exist positive constants $c_1$ and $c_2$ such that 
for any $r \in (0, 1)$ and $u > c_1 r$,
\[ \P\left\{ \sup_{t \in B_\rho(a, r)} |X(t) - X(a)| \ge u \right\} \le \exp\left(- c_2 (u/r)^2\right). \]
Taking $u_n = 2^{-n}\sqrt{2c_2^{-1} \log n}$, we have
\[ \P\left\{ \sup_{t \in B_n} |X(t) - X(a)| \ge u_n \right\} \le n^{-2}. \]
It follows from the Borel--Cantelli lemma that almost surely, for all $n$ large,
\begin{equation}\label{ineq:X}
\sup_{t \in B_n} |X(t) - X(a)| \le u_n.
\end{equation}

(2). Let $\theta_n = 2^{-n}(\log\log 2^n)^{-K}$, where $K > 1$ is the constant in \eqref{Lem2:Eq1}. Define
\begin{equation*}
G_n = \left\{ x \in \R^d : |x| \le  u_n \textup{ with } x = \theta_n p \textup{ for some } p \in \Z^d \right\}.
\end{equation*}
Then, when $n$ is large enough, the cardinality of $G_n$ satisfies
\[ \# (G_n) \le C (\log n)^{(K+1)d}. \]
It follows from \eqref{Lem2:Eq1} of Lemma \ref{Lem2} (with $b = 2$) that we can find a finite constant $c$ 
such that for all $n$ large, 
\[ \P\left\{ \max_{x \in G_n} L(x + X(a), B_n) \ge c\, \varphi(2^{-n}) \right\} \le C(\log n)^{(K+1)d}\, n^{-2}. \]
By the Borel--Cantelli lemma, almost surely, for all $n$ large, 
\begin{equation}\label{max_LT}
\max_{x \in G_n} L(x + X(a), B_n) \le c\, \varphi(2^{-n}).
\end{equation}

(3). For integers $n, k \ge 1$ and $x \in G_n$, define
\begin{equation*}
F(n,k,x) = \bigg\{ y \in \R^d : y = x + \theta_n \sum_{j=1}^k \eps_j 2^{-j}, \eps_j \in \{0, 1\}^d \textup{ for } 1 \le j \le k \bigg\}.
\end{equation*}
A pair of points $y_1, y_2 \in F(n,k,x)$ is said to be linked if $y_2 - y_1 = \theta_n \eps 2^{-k}$ for some $\eps \in \{0, 1\}^d$.
Consider the event $F_n$ defined by
\begin{align*}
F_n = \bigcup_{x \in G_n} \bigcup_{k\ge 1} \bigcup_{y_1, y_2} \Big\{ |L(y_1 + X(a), B_n) - L(y_2 + X(a), B_n)|
 \ge c\,2^{-n(Q-d-\gamma)} |y_1 - y_2|^{\gamma}(k\log n)^{d/Q + K\gamma/Q}\Big\}
\end{align*}
where $\bigcup_{y_1, y_2}$ denotes the union over all linked pairs $y_1, y_2 \in F(n,k,x)$, the constant $c$ is given 
by Lemma \ref{Lem2} with $b = 2$ and a small $\gamma \in (0, 1)$ is chosen so that \eqref{Lem2:Eq2} holds. Note 
that there are at most $2^{kd} 3^d$ linked pairs in $F(n,k,x)$.
It follows that for $n$ large,
\begin{align}
\begin{aligned}\label{P(F_n)}
\P(F_n) &\le C (\log n)^{(K+1)d} \sum_{k = 1}^\infty 2^{kd} \exp(-2k \log n)\\
& = C (\log n)^{(K+1)d} \frac{2^d n^{-2}}{1 - 2^d n^{-2}}.
\end{aligned}
\end{align}
Since $\sum_{n=1}^\infty \P(F_n) < \infty$, it follows from the Borel--Cantelli lemma that a.s.~$F_n$ occurs only finitely many times.

(4). For $y \in \R^d$ with $|y| \le u_n$, $n \ge 1$, we can represent $y$ in the form $y = \lim_{k \to \infty} y_k$ with
\begin{equation}\label{rep:y_k}
y_k = x + \theta_n \sum_{j=1}^k \eps_j 2^{-j},
\end{equation}
where $y_0 = x \in G_n$ and $\eps_j \in \{0, 1\}^d$ for $j = 1, \dots, k$. Since the local time $L$ is jointly continuous, by 
\eqref{rep:y_k} and the triangle inequality, we see that on the event $F_n^c$,
\begin{align}
\begin{aligned}\label{LT:tri_ineq}
|L(y + X(a), B_n) - L(x + X(a), B_n)|
& \le \sum_{k=1}^\infty |L(y_k + X(a), B_n) - L(y_{k-1}+ X(a), B_n)|\\
& \le \sum_{k=1}^\infty c\,2^{-n(Q-d-\gamma)} |y_k - y_{k-1}|^{\gamma} (k \log n)^{d/Q + K\gamma/Q}\\
&\le C \varphi(2^{-n}).
\end{aligned}
\end{align}
Combining \eqref{max_LT} and \eqref{LT:tri_ineq}, we get that a.s.~for all $n$ large, 
\begin{equation}\label{sup_LT}
\sup_{|y| \le u_n} L(y+ X(a), B_n) \le C \varphi(2^{-n}).
\end{equation}
Since $L^*(B_n) = \sup\{ L(x, B_n) : x \in \overline{X(B_n)} \}$, we can deduce \eqref{Holder1} from \eqref{sup_LT} 
and \eqref{ineq:X}. This proves \eqref{Eq:Holder1}.

The proof of \eqref{Eq:Holder2} is similar. It suffices to prove that
\begin{equation}\label{Holder2}
\limsup_{n \to \infty} \sup_{B \in \mathscr{B}_n}\frac{L^*(B)}{\varphi(2^{-n})} \le C,
\end{equation}
where, for each $n \ge 1$, $\mathscr{B}_n$ is a covering of $T$ consisting of disjoint anisotropic cubes of side 
lengths $2^{-n/H_1}, \dots, 2^{-n/H_N}$. Note that $\#(\mathscr{B}_n) \le C 2^{nQ}$.
Define
\[ G_n = \left\{ x \in \R^d : |x| \le  n \textup{ with } x = \theta_n p \textup{ for some } p \in \Z^d \right\}. \]
Then we can use Lemma \ref{LT:large_dev} to find a constant $c$ such that a.s.~for all $n$ large, 
\begin{equation}\label{maxmaxLT}
\max_{B \in \mathscr{B}_n} \max_{x \in G_n} L(x, B) \le c\, \Phi(2^{-n}).
\end{equation}
Define $F(n,k,x)$ as before and 
\begin{align*}
F_n = \bigcup_{B \in \mathscr{B}_n} \bigcup_{x \in G_n} \bigcup_{k \ge 1} \bigcup_{y_1, y_2} \Big\{ |L(y_1, B) - L(y_2, B)|
\ge c\, 2^{-n(Q-d-\gamma)} |y_1 - y_2|^{\gamma} (k\log 2^n)^{d/Q + K\gamma/Q} \Big\}.
\end{align*}
As in \eqref{P(F_n)}, we can use \eqref{Lem:Eq2} to show that a.s.~$F_n$ occurs only finitely many times. 
Since $X(t)$ is continuous, there exists $n_0 = n_0(\omega)$ such that
$\sup_{t \in T} |X(t)| \le n_0$ a.s.
If $|y| \le n$, then by the chaining argument as in \eqref{rep:y_k} and \eqref{LT:tri_ineq}, we can prove that on $F_n^c$,
\[ |L(y, B) - L(x, B)| \le C \Phi(2^{-n}) \]
for some $x \in G_n$.
This and \eqref{maxmaxLT} imply that a.s.~for all $n$ large,
\begin{equation}\label{supsupLT}
\sup_{B \in \mathscr{B}_n} \sup_{|y| \le n} L(y, B) \le C \Phi(2^{-n}) \quad \textup{a.s.}
\end{equation}
Since $L(y, T) = 0$ for $|y| > n_0$, \eqref{Holder2} follows from \eqref{supsupLT}. This completes the proof.
\end{proof}

As pointed out by Berman \cite{B73} (see also Ehm \cite{E81}), the H\"older conditions of the local times are 
closely related to the degree of oscillations of the sample paths of $X(t)$. As a consequence of Theorem \ref{Th:Holder} 
and the inequality (\ref{LT_ineq}) below, we obtain lower bounds for Chung's law of iterated logarithm and the 
modulus of non-differentiability for $X(t)$. 

\begin{theorem}\label{T:sample}
Suppose $X$ satisfies condition \textup{(A)} on $T$. 
Then there exist positive constants $C$ and $C'$ such that for any $t \in T$, 
\begin{equation}\label{Eq:Chung_LIL}
\liminf_{r \to 0} \sup_{s \in B_\rho(t, r)} \frac{|X(s) - X(t)|}{r (\log\log (1/r))^{-1/Q}} \ge C \quad \text{a.s.}
\end{equation}
and
\begin{equation}\label{Eq:mod_nondiff}
\liminf_{r \to 0} \inf_{t \in T} \sup_{s \in B_\rho(t, r)} \frac{|X(s) - X(t)|}{r (\log(1/r))^{-1/Q}} \ge C' \quad \text{a.s.}
\end{equation}
In particular, the sample paths of $X$ are a.s.~nowhere differentiable in $T$.
\end{theorem}

\begin{proof}
It is enough to consider the case $d = 1$.
Let $I$ denote the smallest closed interval containing the range of $X$ on $B_\rho(t, r)$. It follows from the 
occupation density formula (\ref{Eq:ODF}) that
\begin{align}
\begin{aligned}\label{LT_ineq}
\lambda_N(B_\rho(t, r)) = \int_I L(x, B_\rho(t, r)) \, dx
\le L^*(B_\rho(t, r)) \times \sup_{s, s' \in B_\rho(t, r)} |X(s) - X(s')|.
\end{aligned}
\end{align}
Since $\lambda_N(B_\rho(t, r)) = C r^Q$, \eqref{Eq:Chung_LIL} follows from \eqref{LT_ineq} and \eqref{Eq:Holder1}. 
Similarly, \eqref{Eq:mod_nondiff} follows from \eqref{LT_ineq} and \eqref{Eq:Holder2}.
\end{proof}

We end this section with the following remark on the optimality of the inequalities in \eqref{LT_ineq}, \eqref{Eq:Holder2},
 \eqref{Eq:Chung_LIL}, and \eqref{Eq:mod_nondiff}.
 
\begin{remark}\label{Re:optimal}
As pointed out in Ehm \cite{E81}, if the left-hand side of \eqref{Eq:Chung_LIL} [or \eqref{Eq:mod_nondiff}, resp.]~is 
also bounded above by a finite constant a.s., then \eqref{LT_ineq} implies that \eqref{Eq:Holder1} [or 
\eqref{Eq:Holder2}, resp.]~is also bounded below by a positive constant a.s.~and hence the H\"older conditions for 
the local time will be optimal. Indeed, if $X$, in addition to satisfying Condition (A), has stationary increments or satisfies 
Assumption 2.1 in \cite{DMX17}, then \eqref{Eq:Chung_LIL} is a.s.~equal to some positive finite constant, see 
 \cite{LX10} and \cite{LeeX21}, respectively. On the other hand, Wang, Su and Xiao \cite{WangSX20} determined 
the exact modulus of non-differentiability for a class of Gaussian random fields with stationary and isotropic increments.
Hence \eqref{Eq:Holder2} is also optimal for these Gaussian random fields.
\end{remark}

\section{The exact Hausdorff measure of level sets}
\label{sect:level_set}

Let us consider the class $\mathscr{C}$ of functions $\varphi: [0, \delta_0] \to \R_+$ such that $\varphi$ is 
nondecreasing, continuous, $\varphi(0) = 0$, and satisfies the doubling condition, i.e. there exists a finite constant 
$c_0 > 0$ such that
\begin{equation}\label{Eq:gauge_function_doubling_cond}
\frac{\varphi(2s)}{\varphi(s)} \le c_0
\end{equation}
for all $s \in (0, \delta_0/2)$.

Let $\varphi \in \mathscr{C}$ and let $\rho$ be a metric on $\R^N$. For any Borel set $A$ in $\R^N$, the 
\emph{Hausdorff measure} of $A$ with respect to the function $\varphi$, in metric $\rho$ is defined by
\begin{align*}
\mathcal{H}^\varphi_\rho(A) = \lim_{\eps \to 0} \inf\Bigg\{ \sum_{n=1}^\infty \varphi(2r_n) : A \subseteq 
\bigcup_{n=1}^\infty B_\rho(t^n, r_n) \text{ where } t^n \in \R^N \text{ and } r_n \le \eps \text{ for all } n \Bigg\}.
\end{align*}
We use the notation $\mathcal{H}^\varphi(A)$ if $\rho$ is the Euclidean metric.

When $\varphi(s) = s^\alpha$, where $\alpha > 0$ is a real number, $\mathcal{H}^\alpha_\rho(A) = \mathcal{H}^\varphi_\rho(A)$ 
is called the $\alpha$-dimensional Hausdorff measure of $A$ in metric $\rho$, and the Hausdorff dimension of $A$ in $\rho$ is 
defined as
\[ \dim_H^\rho(A) = \inf\{ \alpha > 0 : \mathcal{H}^\alpha_\rho(A) = 0 \}. \]
Hausdorff dimension in metric $\rho$ is useful in studying the fractal properties of anisotropic Gaussian random fields; see 
Wu and Xiao \cite{WX07} and Xiao \cite{X09}.

Suppose $X$ satisfies condition (A) on a compact interval $T \subset \R^N$. 
By Theorem 7.1 of Xiao \cite{X09}, for any $x \in \R^d$, if $Q < d$, then $X^{-1}(x) \cap T = \varnothing$ a.s.; if $d < Q$, 
then with positive probability, the Hausdorff dimension of $X^{-1}(x) \cap T$ in the Euclidean metric is 
(assuming that $0< H_1 \le H_2 \le \dots \le H_N < 1$)
\begin{align}
\begin{aligned}\label{Eq:Haus_meas_level_set_Euclidean}
\dim_H(X^{-1}(x) \cap T) &= \min_{1 \le k \le N}\bigg\{ \sum_{j=1}^k \frac{H_k}{H_j} + N - k - H_k d \bigg\}\\
& = \sum_{j=1}^\tau \frac{H_\tau}{H_j} + N - \tau - H_\tau d,
\end{aligned}
\end{align}
where $\tau$ is the unique integer between $1$ and $N$ such that $\sum_{j=1}^{\tau-1} (1/H_j) \le d < \sum_{j=1}^\tau (1/H_j)$.
More generally, Bierm\'e, Lacaux and Xiao \cite{BLX09} determined the Hausdorff dimension of the inverse image $X^{-1}(F)$ for 
Borel sets $F$ in $\R^N$. 

The Hausdorff dimension of the level set may be different when the underlying metric is not the Euclidean metric. 
Theorem 4.2 of Wu and Xiao \cite{WX11} shows that if $d < Q$, then almost surely, the Hausdorff dimension of 
$X^{-1}(x) \cap T$ in the metric $\rho$ defined by \eqref{D:rho} is 
\[ \dim_H^\rho(X^{-1}(x) \cap T) = Q - d \]
for all $x \in \R^d$ such that $L(x, T) > 0$.

For the special case where $H_1 = \dots = H_N = H$, we have $\dim_H(X^{-1}(x) \cap T) = N - Hd$ and 
$\dim_H^\rho(X^{-1}(x) \cap T) = \frac{N}{H}  - d$. In this case, since $\rho(t, s) \asymp |t - s|^H$, it is easy to 
see that for any function $\varphi \in \mathscr{C}$, there exist positive finite constants $C_1$ and $C_2$ such that 
\begin{equation}\label{compare_Haus_meas}
C_1\mathcal{H}^\psi(A) \le \mathcal{H}^\varphi_\rho(A) \le C_2 \mathcal{H}^\psi(A)
\end{equation}
for all $A \in \mathscr{B}(\R^N)$, where $\psi$ is defined by $\psi(r) = \varphi(r^H)$.

The Hausdorff measure with respect to a suitable function provides a way to measure the size of the level sets, 
especially when the level sets have trivial Lebesgue measure or $\alpha$-dimensional Hausdorff measure. 
It would be interesting to determine the exact Hausdorff measure function (or gauge function) of the level set, 
that is, to find a function $\varphi$ such that 
\[ 0 < \mathcal{H}^\varphi_\rho(X^{-1}(x) \cap T) < \infty \]
almost surely or with positive probability.

Recall that the $\rho$-upper $\varphi$-density of a finite Borel measure $\mu$ on $\R^N$ at the point $t \in \R^N$ is defined by
\[ \overline{D}^{\varphi, \rho}_\mu(t) := \limsup_{r \to 0} \frac{\mu(B_\rho(t, r))}{\varphi(r)}. \]
There exists a positive constant $c \ge 1$ depending only on $c_0$ in \eqref{Eq:gauge_function_doubling_cond} such that 
\begin{equation}\label{Eq:upper_density}
c^{-1} \mathcal{H}^\varphi_\rho(E) \inf_{t \in E}\overline{D}^{\varphi, \rho}_\mu(t)
\le \mu(E) \le c\, \mathcal{H}^\varphi_\rho(E) \sup_{t \in E}\overline{D}^{\varphi, \rho}_\mu(t)
\end{equation}
for any finite Borel measure $\mu$ on $\R^N$ and any Borel set $E$ in $\R^N$ (see Theorem 4.1 of \cite{WX11}).

The following is a partial result giving a lower bound for the Hausdorff measure. As shown by Theorem 5.3 below, it is
possible to show $\mathcal{H}^\varphi_\rho(X^{-1}(x) \cap T) < \infty$ under certain extra conditions.   

\begin{theorem}\label{Thm:Haus_meas_level_set}
Suppose $X$ satisfies condition \textup{(A)} on $T$ and $d < Q$, where $Q = \sum_{j=1}^N (1/H_j)$.
Then there is a constant $C > 0$ such that for any $x \in \R^d$, 
\[ C L(x, T) \le \mathcal{H}^\varphi_\rho(X^{-1}(x) \cap T) \quad \text{a.s.,} \quad \text{where }\quad \varphi(r) = r^{Q-d} (\log\log(1/r))^{d/Q}. \]
In particular, if $H_1 = \dots = H_N = H$, then
\[ C L(x, T) \le \mathcal{H}^\psi(X^{-1}(x) \cap T) \quad \text{a.s.,} \quad \text{where }\quad \psi(r) = r^{N-Hd} (\log\log(1/r))^{Hd/N}. \]
\end{theorem}

\begin{proof}
Take $\mu = L(x, \cdot \cap T)$, which is a.s.~a finite Borel measure on $\R^N$ whose support is $X^{-1}(x) \cap T$. 
By Theorem \ref{Thm:Holder_LT}, there exists a finite constant $C$ such that 
\[ \sup_{t \in E}\overline{D}^{\varphi, \rho}_\mu(t) \le C \quad \text{a.s.} \]
Then we can use the upper bound of \eqref{Eq:upper_density} with $E = X^{-1}(x) \cap T$ to obtain the result. 
The special case where $H_1 = \dots = H_N = H$ follows from \eqref{compare_Haus_meas}
\end{proof}

\smallskip

In view of Theorem \ref{Thm:Haus_meas_level_set}, a natural question is whether  $X^{-1}(x) =\varnothing $ a.s. 
when $d \ge Q$.  As shown by Dalang et al.~\cite[Theorem 2.6]{DMX17}, this is indeed the case if, in addition to satisfying 
conditions of Theorem \ref{Thm:Haus_meas_level_set}, $X$ also satisfies Assumptions 2.1 and 2.4 in \cite{DMX17}.

\section{Systems of stochastic heat equations}
\label{sect:SHE}

As an example, we consider the following system of stochastic heat equations:
\begin{align}\label{Eq:SHE}
\begin{cases}
\displaystyle\frac{\partial}{\partial t} u_j(t, x) = \Delta u_j(t, x) + \dot{W}_j(t, x), & t \ge 0, x \in \R^N,\\
u_j(0, x) = 0 & j = 1, \dots, d,
\end{cases}
\end{align}
where $\dot{W} = (\dot{W}_1, \dots, \dot{W}_d)$ is a $d$-dimensional Gaussian noise. 
In this section, we will apply our main results to study the local times and level sets of the solution $u$ of \eqref{Eq:SHE}. 
At the end of this section, we will also discuss a more general system \eqref{Eq:CSHE} where each component 
of the solution may depend on all the $\dot{W}_j$'s ($j = 1, \dots, d$).

We assume that $\dot{W}_1, \dots, \dot{W}_d$ are i.i.d.~and $\dot{W}_j(t, x)$ is either 
(i) white in time and colored in space with covariance
\[ \E[\dot{W}_j(t, x) \dot{W}_j(s, y)] = \delta_0(t-s) |x-y|^{-\beta} \]
for $N \ge 1$ and $0 < \beta < 2 \wedge N$, or (ii) the space-time white noise for $N = 1$ (take $\beta=1$ in this case).

The solution of \eqref{Eq:SHE} is the Gaussian random field $u = \{u(t, x) : t \ge 0, x \in \R^N\}$ with i.i.d.~components 
$u_1, \dots, u_d$, given by
\[ u_j(t, x) = \int_0^t \int_{\R^N} G(t-s, x-y) W_j(ds\, dy), \]
where $G$ is the fundamental solution of the heat equation:
\[ G(t, x) = \frac{1}{(4\pi t)^{N/2}} \exp\Big(-\frac{|x|^2}{4t}\Big) {\bf 1}_{\{t > 0\}}. \]

Recall that for any $0 < a < b < \infty$, there exist positive finite constants $C_1, C_2$ such that
\begin{equation}\label{Eq:SHE_incre}
C_1 \rho((t, x), (s, y)) \le (\E|u(t, x) - u(s, y)|^2)^{1/2} \le C_2 \rho((t, x), (s, y))
\end{equation}
for all $(t, x), (s, y) \in [a, b] \times [-b,b]^N$, where
\begin{align}\label{rho}
\rho((t, x), (s, y)) = |t-s|^{\frac{2-\beta}{4}} + |x-y|^{\frac{2-\beta}{2}}.
\end{align}
See \cite{DKN07}, Lemma 4.2.
Hence $u$ satisfies (A1) on any compact interval in $(0, \infty) \times \R^N$ and, in this case, the quantity $Q$ becomes
$Q = \frac{2(2+N)}{2 -\beta}$.

We are going to prove that $u$ also satisfies condition (A2), the strong local nondeterminism in variables $(t, x)$ jointly 
with respect to the metric $\rho$. In fact, using the idea of string process in Mueller and Tribe \cite{MT02}, it is shown 
in Herrell et al.~\cite{HSWX} that $u$ admits the decomposition
\begin{equation} \label{Eq:String}
u(t, x) = U(t, x) - Y(t, x),
\end{equation}
where $U(t, x)$ has stationary increments and satisfies strong LND in metric $\rho$, whereas $Y(t, x)$ is a.s.
~continuously differentiable in $(t, x) \in [0, \infty)\times \R^N$. We mention that, by applying the decomposition 
(\ref{Eq:String}) and the stationarity of the increments of $U(t, x)$, Herrell et al.~\cite{HSWX} proved the regularity 
properties such as the exact uniform and local moduli of continuity and Chung's law of the iterated logarithm for 
$u(t, x)$. Lee and Xiao \cite{LeeX21} showed that the regularity properties such as those studied in \cite{HSWX}
can be established under the more general framework of Dalang et al \cite{DMX17} for Gaussian random fields whose 
increments may not be stationary.

In Proposition \ref{Prop:SHE_SLND} below, we will prove directly that $u$ itself satisfies the strong 
LND property in (A2). 

Let $n \ge 1$, $(t^1, x^1), \dots, (t^n, x^n) \in \R_+ \times \R^N$ and $a_1, \dots, a_n \in \R$. Let
\[g(s, y) = \sum_{j=1}^n a_j G(t^j - s, x^j - y) {\bf 1}_{[0, t^j]}.\]
Then by Plancherel's theorem, we have
\begin{align}
\begin{aligned}\label{Eq:SHE_var}
\E\Bigg[ \bigg( \sum_{j=1}^n a_j u_1(t^j, x^j) \bigg)^2 \Bigg]
= C \int_{\R} d\tau \int_{\R^N} d\xi \, |\mathscr{F} g(\tau, \xi)|^2\, |\xi|^{\beta - N}.
\end{aligned}
\end{align}
In the above, $\mathscr{F}g$ denotes the Fourier transform of $g$, that is,
\[ \mathscr{F} g(\tau, \xi) = \int_{\R}\int_{\R^N} e^{-i\tau s - i \langle \xi, y \rangle} g(t, x) \,ds\, dy. \]
One can directly verify that
\begin{equation}\label{FT_G}
\mathscr{F}(G(t - \cdot, x- \cdot) {\bf 1}_{[0, t]})(\tau, \xi) 
= e^{-i\langle \xi, x\rangle} \frac{e^{-i\tau t} - e^{-t|\xi|^2}}{|\xi|^2 - i\tau}.
\end{equation}

\begin{proposition}\label{Prop:SHE_SLND}
For any $0 < a < b < \infty$, there exists a constant $C > 0$ such that
for all integers $n \ge 1$ and all $(t, x), (t^1, x^1), \dots, (t^n, x^n) \in [a, b] \times [-b, b]^N$, 
\begin{equation}\label{Eq:SHE_SLND}
\mathrm{Var}\left(u_1(t, x) | u_1(t^1, x^1), \dots, u_1(t^n, x^n)\right) \ge C \min_{1 \le i \le n}\rho((t, x), (t^i, x^i))^2,
\end{equation}
where $\rho$ is the metric in \eqref{rho}.
\end{proposition}

\begin{proof}
Since $u$ is Gaussian, the conditional variance in \eqref{Eq:SHE_SLND} is the squared $L^2(\P)$-distance 
from $u_1(t, x)$ to the subspace generated by 
$u_1(t^1, x^1), \dots, u_1(t^n, x^n)$ in $L^2(\P)$, that is,
\[ \mathrm{Var}\left(u_1(t, x) | u_1(t^1, x^1), \dots, u_1(t^n, x^n)\right) 
= \inf_{a_1, \dots, a_n \in \R} \E\Bigg[ \bigg( u_1(t, x) - \sum_{j=1}^n a_j u_1(t^j, x^j)\bigg)^2 \Bigg]. 
\]
Therefore, it suffices to show that there exists a positive constant $C$ such that for all $n \ge 1$, for all $(t, x)$, $(t^1, x^1)$, $\dots$, $(t^n, x^n) \in [a, a'] \times [-b, b]^N$, 
and for all $a_1, \dots, a_n \in \R$, we have
\[\E\Bigg[ \bigg( u_1(t, x) - \sum_{j=1}^n a_j u_1(t^j, x^j) \bigg)^2 \Bigg] \ge Cr^{2-\beta},\]
where
\[r = \min_{1 \le j \le n}(|t - t^j|^{1/2} \vee |x - x^j|). \]
From \eqref{Eq:SHE_var} and \eqref{FT_G}, we have
\begin{align}
\label{Eq:var_integral}
&\quad \E\Bigg[ \bigg( u_1(t, x) - \sum_{j=1}^n a_j u_1(t^j, x^j) \bigg)^2 \Bigg]\\
& = C \int_{\R} d\tau \int_{\R^N} d\xi \, 
\bigg|e^{-i\langle \xi, x\rangle}(e^{-i\tau t} - e^{-t |\xi|^2}) 
- \sum_{j=1}^n a_j e^{-i\langle \xi, x^j\rangle}(e^{-i\tau t^j} - e^{-t^j |\xi|^2})\bigg|^2 \frac{|\xi|^{\beta-N}}{|\xi|^{4} + |\tau|^2}.\notag
\end{align}

Let $M$ be such that $|t - s|^{1/2} \vee |x-y| \le M$ for all $(t, x), (s, y) \in [a,a'] \times [-b, b]^N$.
Let $h = \min\{a/M^2, 1\}$.
Let $\varphi: \R \to \R$ and $\psi: \R^N \to \R$ be nonnegative smooth test functions  
that vanish outside the interval $(-h, h)$ and the unit ball respectively and 
satisfy $\varphi(0) = \psi(0) = 1$.
Let $\varphi_r(\tau) = r^{-2} \varphi(r^{-2}\tau)$ 
and $\phi_r(\xi) = r^{-N} \psi(r^{-1} \xi)$.
Consider the integral
\begin{align*}
I := 
\int_{\R} d\tau \int_{\R^N} d\xi \bigg[e^{-i\langle \xi, x\rangle}(e^{-i\tau t} - e^{-t |\xi|^2}) 
- \sum_{j=1}^n a_j e^{-i\langle \xi, x^j\rangle}(e^{-i\tau t^j} - e^{-t^j |\xi|^2})\bigg]
e^{i\langle \xi, x\rangle} e^{i\tau t} \widehat{\varphi}_r(\tau) \widehat{\psi}_r(\xi).
\end{align*}
By inverse Fourier transform, we have
\begin{align*}
I = (2\pi)^{1+N}\bigg[\varphi_r(0) \psi_r(0) - \varphi_r(t) (p_t \ast \psi_r)(0)
- \sum_{j=1}^n a_j \Big( \varphi_r(t-t^j) \psi_r(x-x^j) - 
\varphi_r(t) (p_{t^j} \ast \psi_r)(x-x^j)\Big)\bigg],
\end{align*}
where $p_t(x) = G(t, x)$ is the heat kernel.
By the definition of $r$, $|t-t^j| \ge r^2$ or $|x-x^j| \ge r$ for each $j \in \{1, \dots, n\}$, 
thus $\varphi_r(t-t^j) \psi_r(x-x^j) = 0$. Moreover, since $t/r^2 \ge a/M^2 \ge h$, 
we have $\varphi_r(t) = 0$ and hence
\begin{equation}\label{Eq:I}
I = (2\pi)^{1+N} r^{-2-N}.
\end{equation}
On the other hand, by the Cauchy--Schwarz inequality and \eqref{Eq:var_integral},
\begin{align*}
I^2 \le C\, \E\Bigg[ \bigg( u_1(t, x) - \sum_{j=1}^n a_j u_1(t^j, x^j) \bigg)^2 \Bigg] 
\int_{\R} \int_{\R^N} \big|\widehat{\varphi}_r(\tau) \widehat{\psi}_r(\xi)\big|^2 
\big(|\xi|^{4}+|\tau|^2\big) |\xi|^{N-\beta} d\tau \, d\xi.
\end{align*}
Note that $\widehat{\varphi}_r(\tau) = \widehat{\varphi}(r^2 \tau)$ and 
$\widehat{\psi}_r(\xi) = \widehat{\psi}(r\xi)$. Then by scaling,
\begin{align*}
\int_{\R} \int_{\R^N} \big|\widehat{\varphi}_r(\tau) \widehat{\psi}_r(\xi)\big|^2 
\big(|\xi|^{4}+|\tau|^2\big) |\xi|^{N-\beta} d\tau \, d\xi
 = r^{-6+\beta-2N} \int_{\R} \int_{\R^N} \big|\widehat{\varphi}(\tau) \widehat{\psi}(\xi)\big|^2 
\big(|\xi|^{4}+|\tau|^2\big)|\xi|^{N-\beta} d\tau \, d\xi.
\end{align*}
The last integral is finite since $\widehat{\varphi}$ and $\widehat{\psi}$ are 
rapidly decreasing functions. It follows that
\begin{equation}\label{Eq:I^2}
I^2 \le C_0 r^{-6 +\beta -2N} \,
\E\Bigg[ \bigg( u_1(t, x) - \sum_{j=1}^n a_j u_1(t^j, x^j) \bigg)^2 \Bigg]
\end{equation}
for some finite constant $C_0$.
Combining \eqref{Eq:I} and \eqref{Eq:I^2}, we get that
\[ \E\Bigg[ \bigg( u_1(t, x) - \sum_{j=1}^n a_j u_1(t^j, x^j) \bigg)^2 \Bigg] \ge (2\pi)^{2+2N}C_0^{-1} r^{2-\beta}. \]
The proof is complete.
\end{proof}

We have shown that $u$ satisfies condition (A) on any compact interval $T$ in $(0, \infty) \times \R^N$.
Therefore, the following result is a direct consequence of Theorems \ref{Th:joint_cont_LT}, \ref{Thm:Holder_LT},
and \ref{Th:Holder}.

\begin{corollary}\label{cor:LT}
Suppose $d < Q := \frac{2(2+N)}{2-\beta}$ and $T$ is any compact interval in $(0, \infty) \times \R^N$.
Then $u(t, x)$ has a jointly continuous local time $L(z, T)$ on $T$ satisfying the H\"older conditions 
\eqref{Eq:Holder_LT}, \eqref{Eq:Holder1} and \eqref{Eq:Holder2}.
\end{corollary}

Finally, we consider the level sets $u^{-1}(z) \cap T = \{ (t, x) \in T : u(t, x) = z \}$, where $z \in \R^d$.
Recall from Section \ref{sect:level_set} that $u^{-1}(z) \cap T = \varnothing$ a.s.~if $Q < d$. The same 
is true when $Q = d$, which was proved by Dalang, Mueller and Xiao \cite{DMX17}. 

If $d < Q$, we are able to obtain a precise result for the level sets of $u$.
The following theorem determines the exact Hausdorff measure function for the level sets.

\begin{theorem}\label{Thm:SHE_Haus_meas}
Suppose $d < Q$. Let $T$ be a compact interval in $(0, \infty) \times \R^N$. 
Then there exists a constant $C > 0$ such that for any $z \in \R^d$,
\begin{equation}\label{Eq:SHE_Haus_meas}
CL(z, T) \le \mathcal{H}_\rho^\varphi(u^{-1}(z) \cap T) < \infty \quad \text{a.s.}
\end{equation}
where $\varphi(r) = r^{Q-d} (\log\log(1/r))^{d/Q}$.
\end{theorem}

\begin{remark}
We conjecture that (\ref{Eq:SHE_Haus_meas}) can be strengthened to: There exist positive finite constants 
$C_1$ and $C_2$ such that
\[ C_1L(z, T) \le \mathcal{H}_\rho^\varphi(u^{-1}(z) \cap T) \le C_2 L(z, T) \quad \text{a.s.} \]
\end{remark}

\begin{remark}
Let $\delta((t, x),(s,y)) = |t-s|^{1/2} + |x-y|$ be the parabolic metric in $\R^{1+N}$. Since 
$\rho((t, x), (s, y)) \asymp [\delta((t, x),(s,y))]^{(2-\beta)/2}$, it follows from Theorem \ref{Thm:SHE_Haus_meas} 
that
\[ CL(z, T) \le \mathcal{H}^\psi_\delta(u^{-1}(z)\cap T) < \infty \quad \text{a.s.}
\]
where $\psi(r) = r^{2+N-d(2-\beta)/2} (\log\log(1/r))^{d/Q}$. In particular, the parabolic Hausdorff 
dimension of the level set is $2+N - d(2-\beta)/2$.
\end{remark}

\begin{proof}[Proof of Theorem \ref{Thm:SHE_Haus_meas}]
The lower bound in \eqref{Eq:SHE_Haus_meas} follows immediately from Theorem \ref{Thm:Haus_meas_level_set}.
To prove that the Hausdorff measure is finite, we use the method in Xiao \cite{X97}, which is similar to 
Talagrand's covering argument in \cite{T98}. 
First, we may assume $T = B((t_0, x_0), \eta_0)$, where $\eta_0 > 0$ is small and $(t_0, x_0) \in T$ are fixed. Let
\begin{align*}
u^1(t, x) = u(t, x) - u^2(t, x) \quad \text{and} \quad
u^2(t, x) = \E(u(t, x) | u(t_0, x_0)).
\end{align*}
Then $u^1$ and $u^2$ are independent processes.

For our current proof, it would be easier for us to work with ``cubes'' that are comparable to balls in metric $\rho$ and, 
at the same time, have the nested property of the ordinary dyadic cubes.
For this reason, we are going to use a family of generalized dyadic cubes $\mathscr{Q}$, which can be obtained by 
Theorem 2.1 and Remark 2.2 of \cite{KRS12} applied to the metric space $(T, \rho)$.
More specifically, $\mathscr{Q} = \bigcup_{q=1}^\infty \mathscr{Q}_q$, where $\mathscr{Q}_q = \{ I_{q, \ell} : 
\ell =1, \dots, n_q \}$ are families of Borel subsets of $T$, and there exist constants $c_1, c_2$ such that the 
following properties hold:
\begin{enumerate}
\item[(i)] $T = \bigcup_{\ell=1}^{n_q} I_{q, \ell}$ for each $q \ge 1$;
\item[(ii)] Either $I_{q, \ell} \cap I_{q', \ell'} = \varnothing$ or $I_{q, \ell}
\subset I_{q', \ell'}$ whenever $q \ge q'$, $1 \le \ell \le n_q$, $1 \le \ell' \le n_{q'}$;
\item[(iii)] For each $q, \ell$, there exists $x_{q, \ell} \in T$ such that
$ B_\rho(x_{q,\ell}, c_1 2^{-q}) \subset I_{q,\ell} \subset B_\rho(x_{q,\ell}, c_2 2^{-q})$
and $\{ x_{q,\ell} : 1, \dots, n_q \} \subset \{x_{q+1, \ell} : \ell = 1, \dots, n_{q+1} \}$ for
all $q \ge 1$.
\end{enumerate}
For simplicity, any member of $\mathscr{Q}_q$ will be called a dyadic cube of order $q$.

The main ingredient for the covering argument is the following estimate: 
there exist a finite constant $K_1$ and $\eta_1 > 0$ small such that for all $0 < r_0 < \eta_1$, and all $(t, x) \in T$, we have
\begin{align}
\begin{aligned}\label{Eq:SHE_main_est}
\P\left\{ \exists\, r \in [r_0^2, r_0], \sup_{(s, y) \in B_\rho((t, x), 2c_2r)} |u(t, x) - u(s, y)| \le K_1 r\Big( \log\log \frac{1}{r}\Big)^{-1/Q} \right\}
\ge 1 - \exp\bigg(-\Big(\log\frac{1}{r_0}\Big)^{1/2} \bigg).
\end{aligned}
\end{align}
This is proved for a more general class of Gaussian random fields in Dalang et al.~\cite{DMX17}.

Moreover, by Lemma 5.3 and Lemma 7.5 of \cite{DMX17}, there exists a finite constant $K_2$ such that for all $(t, x), (s, y) \in T$, 
\begin{equation}\label{Eq:SHE_lemma4}
|u^2(t, x) - u^2(s, y)| \le K_2 \big(|t-s| + \sum_{j=1}^N|x_j-y_j|\big) |u(t_0, x_0)|.
\end{equation}
Let 
\begin{align*}
R_p = \Bigg\{ (t, x) \in T : \,&\exists\, r \in [2^{-2p}, 2^{-p}] \text{ such that }
\sup_{(s, y) \in B_\rho((t, x), 2c_2r)} |u(t, x) - u(s, y)| \le K_1r \left( \log\log\frac{1}{r}\right)^{-1/Q} \Bigg\}.
\end{align*}
Consider the events
\begin{align*}
\Omega_{p, 1} &= \Big\{ \omega : \lambda_{N+1}(R_p) \ge \lambda_{N+1}(T) (1 - \exp(-\sqrt{p}/4)) \Big\},\\
\Omega_{p, 2} &= \Big\{ \omega : |u(t_0, x_0)| \le 2^{pb} \Big\},
\end{align*}
where $b > 0$ is chosen and fixed such that $\frac{2}{2-\beta} - b > 1$.
By \eqref{Eq:SHE_main_est}, 
\begin{align}\label{E:P:(t,x)inRp}
\P\{(t, x) \in R_p\} \ge 1 - \exp(-\sqrt{p/2}).
\end{align}
Then by Markov's inequality, Fubini's theorem and \eqref{E:P:(t,x)inRp}, we have
\begin{align*}
\P(\Omega_{p,1}^c) &= \P\{ \lambda_{N+1}(T\setminus R_p) \ge \lambda_{N+1}(T) \exp(-\sqrt{p}/4) \}\\
&\le \frac{1}{\lambda_{N+1}(T) \exp(-\sqrt{p}/4)} \E[\lambda_{N+1}(T\setminus R_p)]\\
&= \frac{1}{\lambda_{N+1}(T) \exp(-\sqrt{p}/4)} \int_T \P\{(t,x) \not\in R_p\} \lambda_{N+1}(dt\,dx) \\
&\le \frac{1}{\lambda_{N+1}(T) \exp(-\sqrt{p}/4)} \exp(-\sqrt{p/2}).
\end{align*}
Hence, it follows that $\sum_{p=1}^\infty \P(\Omega_{p,1}^c) < \infty$.
Moreover, it is easy to see that $\sum_{p=1}^\infty \P(\Omega_{p,2}^c) < \infty$.
Consider the event
\begin{align*}
\Omega_{p, 3} = \bigg\{ \omega : \forall \, I \in \mathscr{Q}_{2p}, \sup_{(t, x), (s, y) \in I} |u(t, x) - u(s, y)|
 \le K_2 2^{-2p} (\log 2^{2p})^{1/2} \bigg\}.
\end{align*}
By Lemma 2.1 of Talagrand \cite{T95} (see also Lemma 3.1 in \cite{DMX17}), we see that for $K_2$ and $p$ large, 
\[ \P\left\{\sup_{(t, x), (s, y) \in I} |u(t, x) - u(s, y)| \le K_2 2^{-2p} (\log 2^{2p})^{1/2} \right\}
\le \exp\left( -\left(\frac{K_2}{c_2} \right)^2 p \right). \]
Since the cardinality of the family $\mathscr{Q}_{2p}$ is at most $C 2^{2pQ}$, we have
$\sum_{p=1}^\infty \P(\Omega_{p,3}^c) < \infty$ provided $K_2$ is chosen to be a sufficiently large constant.
Let $\Omega_p = \Omega_{p,1} \cap \Omega_{p,2} \cap \Omega_{p,3}$.
Then
\begin{align}\label{E:P_Omega}
\P(\Omega^*) = 1, \quad \text{where } \Omega^* := \bigcup_{\ell \ge 1} \bigcap_{p \ge \ell} \Omega_p.
\end{align}
Moreover, we define
\begin{align*}
R_p' = \Bigg\{ (t, x) \in T : \,\exists\, r \in [2^{-2p}, 2^{-p}] \text{ such that }
 \sup_{(s, y) \in B_\rho((t, x), 2c_2r)} |u^1(t, x) - u^1(s, y)| \le 2K_1r \left( \log\log\frac{1}{r}\right)^{-1/Q} \Bigg\}
\end{align*}
and the event
\begin{align*}
\Omega_{p, 4} = \left\{ \omega : \lambda_{N+1}(R_p') \ge \lambda_{N+1}(T)(1-\exp(-\sqrt{p}/4)) \right\}.
\end{align*}
Note that \eqref{Eq:SHE_lemma4} implies that $R_p \subset R_p'$ on $\Omega_{p, 3}$ for $p$ large enough and hence
\begin{align}\label{E:Omega_p4}
\Omega_{p, 1} \cap \Omega_{p, 3} \subset \Omega_{p, 4}.
\end{align}

We are going to construct a random covering for the level set $u^{-1}(z) \cap T$. For any $p \ge 1$ and $(t, x) \in T$, 
let $I_p(t, x) \in \mathscr{Q}_p$ be the unique dyadic cube of order $p$ containing $(t, x)$. 
We say that $I_q(t, x)$ is a good dyadic cube of order $q$ if it satisfies the following property:
\begin{equation}\label{Eq:good_dyadic_cube}
\sup_{(s, y), (s',y') \in I_q(t, x)} |u^1(s, y) - u^1(s',y')| \le 8 K_1 2^{-q}(\log\log 2^q)^{-1/Q}. 
\end{equation}
For each $(t, x) \in R_p'$, there is some $r \in [2^{-q}, 2^{-q+1}]$ with $p+1 \le q \le 2p$ such that
\begin{align*}
\sup_{(s, y) \in B_\rho((t, x), 2c_2r)}|u^1(t, x) - u^1(s, y)| \le 2K_1 r\left(\log\log\frac{1}{r}\right)^{-1/Q}.
\end{align*}
By property (iii) above, $I_q(t, x)$ is contained in some ball $B_\rho((t^*,x^*),c_2 2^{-q})$, and thus by triangle 
inequality, we have
$I_q(t,x) \subset B_\rho((t,x),2c_2 2^{-q})$ and
\begin{align*}
&\sup_{(s, y), (s', y') \in I_q(t, x)}|u^1(s, y) - u^1(s', y')|\\
&\le \sup_{(s, y) \in B_\rho((t, x), 2c_2r)}|u^1(s, y) - u^1(t, x)| + \sup_{(s', y') \in B_\rho((t, x), 2c_2r)}|u^1(t, x) - u^1(s', y')|\\
& \le 4K_1 r\left(\log\log\frac{1}{r}\right)^{-1/Q} \le 8K_1 2^{-q}(\log\log 2^q)^{-1/Q}.
\end{align*}
Hence, $I_q(t, x)$ is a good dyadic cube of order $q$.
By property (ii), we obtain in this way a family $\mathscr{G}_p^1$ of disjoint dyadic cubes that cover $R_p'$.
On the other hand, we let $\mathscr{G}_p^2$ be the family of dyadic cubes in $T$ of order $2p$ that are not 
contained in any cube of $\mathscr{G}_p^1$. In particular, the cubes in $\mathscr{G}_p^2$ are contained in $T \setminus R_p'$. 
Let $\mathscr{G}_p = \mathscr{G}_p^1 \cup 
\mathscr{G}_p^2$. Note that $\mathscr{G}_p$ depends only on the random field $\{ u^1(t, x), (t, x) \in T\}$.

For each dyadic cube $I \in \mathscr{Q}$, choose a fixed point in $I \cap T$ and label it by $(t_I, x_I)$.
For any $I \in \mathscr{Q}_q$ of order $q$, where $p \le q \le 2p$, consider the event
\[ \Omega_{p,I} = \{ \omega :  |u(t_I, x_I) - z| \le 2 r_{p,I} \} \]
where
\[ r_{p,I} = \begin{cases}
8 K_1 2^{-q}(\log\log2^q)^{-1/Q} & \text{if } I \in \mathscr{G}_p^1 \text{ and } I \text{ is of order } q,\\
K_2 2^{-2p} (\log 2^{2p})^{1/2} & \text{if } I \in \mathscr{G}_p^2.
\end{cases} \]
Let $\mathscr{F}_p$ be the subcover of $\mathscr{G}_p$ (depending on $\omega$) defined by
\[ \mathscr{F}_p(\omega) = \{ I \in \mathscr{G}_p(\omega) : \omega \in \Omega_{p,I} \}. \]

We claim that for $p$ large, on the event $\Omega_p$, $\mathscr{F}_p$ covers the set $u^{-1}(z) \cap T$.
Suppose $\Omega_p$ occurs and $(t, x) \in u^{-1}(z) \cap T$. 
Since $\mathscr{G}_p$ covers $T$, the point $(t, x)$ is 
contained in some dyadic cube $I$ and either $I \in \mathscr{G}_p^1$ or $I \in \mathscr{G}_p^2$. 

Case 1: if $I \in \mathscr{G}_p^1$, then $I = I_q(t, x)$ is a good dyadic cube of order $q$, where $p \le q \le 2p$, 
and \eqref{Eq:good_dyadic_cube} holds. 
Recall that $I$ is contained in some ball $B_\rho(c_2 2^{-q})$.
Since $\Omega_{p,2}$ occurs, it follows that from \eqref{Eq:SHE_lemma4} and \eqref{Eq:good_dyadic_cube} that
\begin{align*}
|u(t_I, x_I) - z| &\le |u^1(t_I, x_I) - u^1(t, x)| + |u^2(t_I, x_I) - u^2(t, x)|\\
& \le 8 K_12^{-q}(\log\log2^q)^{-1/Q} + K_2 \Big( (2 c_2^{\frac{4}{2-\beta}}+2N c_2^{\frac{2}{2-\beta}}) 2^{-q\frac{2}{2-\beta}} \Big) 2^{pb}.
\end{align*}
This is $\le 2 r_{p, I}$ for $p$ large because $b$ is chosen such that $\frac{2}{2-\beta} - b > 1$.
Hence $I \in \mathscr{F}_p$.

Case 2: if $I \in \mathscr{G}_p^2$, since $\Omega_{p,3}$ occurs, we have
\[ |u(t_I, x_I) - z| = |u(t_I, x_I) - u(t, x)| \le K_2 2^{-2p}p^{1/2}. \] 
In this case, $I \in \mathscr{F}_p$. Hence the claim is proved.

Let $\Sigma_1$ be the $\sigma$-algebra generated by $\{ u^1(t, x) : (t, x) \in T\}$.
To estimate the conditional probability $\P(\Omega_{p,I}|\Sigma_1)$, we use the conditional variance formula and 
\eqref{Eq:SHE_incre} to get 
that for all $(t, x) \in T = B_\rho((t_0, x_0), \eta_0)$, 
\begin{align*}
\mathrm{Var}(u^2(t, x)) &=
\mathrm{Var}(\E(u(t, x)|u(t_0, x_0))) = \mathrm{Var}(u(t, x)) - \E[\mathrm{Var}(u(t, x)|u(t_0, x_0))]\\
& \ge \inf_{(t, x) \in T} \mathrm{Var}(u(t, x)) - C_2 \sup_{(t, x) \in T} \rho^2((t, x),(t_0, x_0))
\end{align*}
which is bounded from below by a positive constant provided $\eta_0 > 0$ is small enough.
This variance lower bound and the fact that the process $u^2$ is Gaussian imply that there is a 
constant $C < \infty$ such that for all $(t, x) \in T$, $v \in \R^d$ and $r > 0$, $\P\{|u^2(t, x) - v| \le r\} \le C r^d$.
Since $u^1$ and $u^2$ are independent, we have
\begin{equation}\label{Eq:cond_prob}
\P(\Omega_{p,I}|\Sigma_1) \le C r_{p,I}^d.
\end{equation}

Now, we estimate the expected value of $\mathcal{H}_\rho^\varphi(u^{-1}(z)\cap T)$.
Let $q\{I\}$ denote the order of $I \in \mathscr{F}_p$. 
By conditioning, \eqref{E:Omega_p4} and \eqref{Eq:cond_prob},
\begin{align*}
\E\left[ {\bf 1}_{\Omega_p} \sum_{I \in \mathscr{F}_p} \varphi(2c_2 2^{-q\{I\}}) \right]
&\le\E\left[ {\bf 1}_{{\Omega}_{p, 4}} \sum_{q=p}^{2p}\sum_{I \in \mathscr{Q}_q} \varphi(2c_22^{-q}) 
{\bf 1}_{\{ I \in \mathscr{G}_p \}} {\bf 1}_{\Omega_{p,I}} \right]\\
& =\E\left[ {\bf 1}_{{\Omega}_{p, 4}} \sum_{q=p}^{2p}\sum_{I \in \mathscr{Q}_q} \varphi(2c_22^{-q}) 
{\bf 1}_{\{ I \in \mathscr{G}_p \}} \E\left({\bf 1}_{\Omega_{p,I}}|\Sigma_1\right)\right]\\
& \le C\, \E \left[ {\bf 1}_{{\Omega}_{p, 4}}\sum_{q=p}^{2p} \sum_{I \in \mathscr{Q}_q}
  \varphi(2c_22^{-q}) r_{p,I}^d {\bf 1}_{\{I \in \mathscr{G}_p\}} \right].
\end{align*}
If $I \in \mathscr{G}_p^1$ is of order $q$, then 
\[ \varphi(2c_22^{-q}) r_{p,I}^d \le C 2^{-q(Q-d)}(\log\log2^q)^{d/Q} 2^{-qd}(\log\log2^q)^{-d/Q} \le C \lambda_{N+1}(I), \]
and these $I$'s are disjoint sets contained in $T$.
If $I \in \mathscr{G}_p^2$, then 
\[ \varphi(2c_22^{-2p}) r_{p, I}^d \le C 2^{-2pQ} (\log\log2^{2p})^{d/Q} p^{d/2}. \]
Note that there are at most $C 2^{2pQ}\exp(-\sqrt{p}/4)$ many such $I$'s on the event $\Omega_{p, 4}$ since 
$T \setminus R_p'$ has Lebesgue measure $\le \exp(-\sqrt{p}/4)$ and each $I \in \mathscr{G}_p^2$ has 
Lebesgue measure $\sim C2^{-2pQ}$.
It follows that
\begin{align*}
\E\left[ {\bf 1}_{\Omega_p} \sum_{I \in \mathscr{F}_p} \varphi(2^{-q\{I\}}) \right]
& \le C\, \E\left[ \sum_{I \in \mathscr{G}_p^1} \lambda_{N+1}(I) + {\bf 1}_{\Omega_{p, 4}} \sum_{I \in \mathscr{G}_p^2} 
2^{-pQ}(\log\log2^{2p})^{d/Q} p^{d/2}  \right]\\
& \le C\left(\lambda_{N+1}(T) + (\log 2p)^{d/Q} p^{d/2}\exp(-\sqrt{p}/4)\right)
\end{align*}
provided $p$ is large. 
Recall that $\mathscr{F}_p$ is a cover for $u^{-1}(z) \cap T$ on $\Omega_p$ for large $p$ and 
each $I$ is contained in a ball of radius $c_2 2^{-q\{I\}}$ in metric $\rho$.
Therefore, by \eqref{E:P_Omega} and Fatou's lemma, 
\begin{align*}
\E\left[ \mathcal{H}^\varphi_\rho(u^{-1}(z) \cap T) \right]
&= \E\left[ {\bf 1}_{\Omega^*}\, \mathcal{H}^\varphi_\rho(u^{-1}(z) \cap T) \right]\\
& \le \liminf_{p \to \infty} \E\left[  {\bf 1}_{\Omega_p} \sum_{I \in \mathscr{F}_p} \varphi(2c_22^{-q\{I\}}) \right]
\le C\lambda_{N+1}(T) < \infty.
\end{align*}
This completes the proof of Theorem \ref{Thm:SHE_Haus_meas}.
\end{proof}

Finally, we consider the solution $\{ v(t, x) = (v_1(t, x), \dots, v_d(t, x))^T, t \ge 0, x \in \R^N \}$ of the system
\begin{align}\label{Eq:CSHE}
\begin{cases}
\displaystyle\frac{\partial}{\partial t} v_i(t, x) = \Delta v_i(t, x) + \sum_{j=1}^d A_{ij} \dot{W}_j(t, x), & t \ge 0, x \in \R^N,\\
v_i(0, x) = 0, & i = 1, \dots, d,
\end{cases}
\end{align}
where $A$ is a non-random $d \times d$ invertible matrix and $(\dot{W}_1, \dots, \dot{W}_d)$ is the $d$-dimensional 
Gaussian noise as defined in \eqref{Eq:SHE}. 
The solution is given by the Gaussian random field
\begin{align*}
v_i(t, x) = \int_0^t \int_{\R^N} G(t-s, x, y) \sum_{j=1}^d A_{ij} W_j(ds, dy), \quad i = 1, \dots, d.
\end{align*}
Hence, $v = Au$, where $u(t, x) = (u_1(t, x), \dots, u_d(t, x))^T$ is the solution of \eqref{Eq:SHE}.
In general, $v(t, x)$ has non-i.i.d.~components. Although our main results, which are based on the i.i.d.~setting, 
cannot be applied directly, we can still make use of the relation $v = Au$ to obtain properties for the local times of 
$v$ from those of the local times of $u$.

Let $T$ be a compact interval in $(0, \infty) \times \R^N$. 
Suppose that $d < Q = \frac{2(2+N)}{2-\beta}$.
Since $u(t, x)$ has a local time $L_u$ on $T$ (by Corollary \ref{cor:LT}) and $A$ is an invertible matrix, it follows 
from the relation $v = Au$ that the occupation measure $\mu^v_T(\cdot) = \lambda_{1+N}\{(t, x) \in T : v(t, x) \in \cdot\, \}$ 
is absolutely continuous with respect to the Lebesgue measure $\lambda_d$ on $\R^d$. 
Therefore, $v(t, x)$ also has a local time $L_v$ on $T$.
By the occupation density formula \eqref{Eq:ODF} and the relation $v = Au$, we obtain the following expression 
for the local time of $v$:
\begin{align}\label{Lv-Lu}
L_v(z, I) = |\det A|^{-1} L_u(A^{-1}z, I),
\end{align}
where $I$ is any interval in $T$.
As a result, all the properties of $L_v(z, I)$ including joint continuity and H\"older conditions can be deduced 
from those of $L_u(z, I)$. Moreover, since $v^{-1}(z) = u^{-1}(A^{-1}z)$, Theorem \ref{Thm:SHE_Haus_meas} and 
\eqref{Lv-Lu} imply that
\begin{align*}
C |\det A| \cdot L_v(z, T) \le \mathcal{H}^\varphi_\rho(v^{-1}(z) \cap T) < \infty \quad \text{a.s.}
\end{align*}
Hence, the exact Hausdorff measure function for the level sets of $v$ is the same as that of $u$.

\medskip

{\bf Acknowledgements.} 
The authors wish to thank Professor Davar Khoshnevisan for stimulating discussions and encouraging the authors to publish this paper.
Y. Xiao was supported in part by the NSF grant DMS-2153846.

\bigskip

\textsc{Cheuk Yin Lee}: School of Science and Engineering, The Chinese University of Hong Kong, Shenzhen, China.\\
E-mail address: \texttt{leecheukyin@cuhk.edu.cn}\\

\textsc{Yimin Xiao}: Department of Statistics and Probability, C413 Wells Hall, Michigan State University, East Lansing, MI 48824, United States.\\
E-mail address: \texttt{xiao@stt.msu.edu}\\

\end{document}